\documentclass[reqno]{amsart}

\usepackage{graphicx}
\usepackage{abstract,amsthm,authblk} 
\usepackage[foot]{amsaddr}

\title{Adaptive Localized Cayley Parametrization\\ for Optimization over Stiefel Manifold}
\author{Keita Kume}
\email{kume@sp.ce.titech.ac.jp}
\author{Isao Yamada}
\email{isao@sp.ce.titech.ac.jp}
\address{Dept. of Information and Communications Engineering, Tokyo Institute of Technology, Tokyo, Japan}
\usepackage{epstopdf}
\usepackage[caption=false]{subfig}
\usepackage[numbers,sort&compress]{natbib}
\bibpunct[, ]{[}{]}{,}{n}{,}{,}

\makeatletter
\def\NAT@def@citea{\def\@citea{\NAT@separator}}
\makeatother

\usepackage{amsmath,amssymb,amsfonts}
\usepackage{textcomp}
\usepackage{xcolor}

\usepackage{here}
\usepackage{amssymb}
\usepackage{ascmac}
\usepackage{bm}
\usepackage{algorithm}
\usepackage{algpseudocode}
\graphicspath{{data/}}
\usepackage{comment}
\usepackage{cases}
\usepackage{mathtools}
\usepackage{enumitem}
\usepackage{mathabx}
\usepackage{hyperref}
\usepackage{stmaryrd}
\usepackage{tikz}
\usetikzlibrary{arrows.meta}
\usetikzlibrary{positioning}
\usetikzlibrary{decorations.pathreplacing,calc}
\usepackage{csvsimple}

\usepackage{etoolbox}
\apptocmd{\lim}{\limits}{}{}
\apptocmd{\liminf}{\limits}{}{}

\usepackage{afterpage}

\usepackage{tabularx}
\usepackage{booktabs}
\mathtoolsset{showonlyrefs=true}

\newcommand{\argmin}{\mathop{\mathrm{argmin}}\limits}

\newcommand{\inprod}[2]{{\langle #1,#2 \rangle}}
\newcommand{\trace}{{\rm Tr}}

\newcommand{\St}{{\rm St}}

\newcommand{\TT}{\mathsf{T}}
\newcommand{\diag}{\mathrm{diag}}
\newcommand{\Skew}{\mathop{\mathrm{S_{kew}}}}
\newcommand{\dbra}[1]{\llbracket #1 \rrbracket}

\makeatletter
\newcommand{\doublewidetilde}[1]{{%
  \mathpalette\double@widetilde{#1}%
}}
\newcommand{\double@widetilde}[2]{%
  \sbox\z@{$\m@th#1\widetilde{#2}$}%
  \ht\z@=.9\ht\z@
  \widetilde{\box\z@}%
}
\makeatother

\theoremstyle{plain}
\newtheorem{theorem}{Theorem}[section]
\newtheorem{lemma}[theorem]{Lemma}

\theoremstyle{definition}
\newtheorem{definition}[theorem]{Definition}
\newtheorem{example}[theorem]{Example}

\newtheorem{problem}[theorem]{Problem}

\theoremstyle{remark}
\newtheorem{remark}[theorem]{Remark}

\newtheorem{fact}[theorem]{Fact}

\DeclarePairedDelimiter{\abs}{\lvert}{\rvert}

\begin{document}

\maketitle
\begin{abstract}
	We present an adaptive parametrization strategy for optimization problems over the Stiefel manifold by using generalized Cayley transforms to utilize powerful Euclidean optimization algorithms efficiently.
	The generalized Cayley transform can translate an open dense subset of the Stiefel manifold into a vector space, and the open dense subset is determined according to a tunable parameter called a center point.
	With the generalized Cayley transform, we recently proposed the naive Cayley parametrization, which reformulates the optimization problem over the Stiefel manifold as that over the vector space.
	Although this reformulation enables us to transplant powerful Euclidean optimization algorithms, their convergences may become slow by a poor choice of center points.
	To avoid such a slow convergence, in this paper, we propose to estimate adaptively 'good' center points so that the reformulated problem can be solved faster.
	We also present a unified convergence analysis, regarding the gradient, in cases where fairly standard Euclidean optimization algorithms are employed in the proposed adaptive parametrization strategy.
	Numerical experiments demonstrate that (i) the proposed strategy succeeds in escaping from the slow convergence observed in the naive Cayley parametrization strategy; (ii) the proposed strategy outperforms the standard strategy which employs a retraction.
\end{abstract}


\noeqref{eq:ICT_O,eq:ICT_O,eq:Cayley,eq:Cayley_inv_origin}

\section{Introduction}
The Stiefel manifold
$\St(p,N):= \{\bm{U} \in \mathbb{R}^{N\times p} \mid \bm{U}^{\TT}\bm{U} = \bm{I}_{p}\}$
is defined for
$(p,N) \in \mathbb{N} \times \mathbb{N}$
with
$p \leq N$,
where
$\bm{I}_{p} \in \mathbb{R}^{p\times p}$
is the
$p\times p$
identity matrix.
In this paper, we consider optimization problems with orthogonality constraints formulated as follows:
\begin{problem}\label{problem:origin}
	Let
	$f:\mathbb{R}^{N\times p }\to \mathbb{R}$,
	be a differentiable function, and let its gradient
	$\nabla f:\mathbb{R}^{N\times p}\to \mathbb{R}^{N\times p}$
	Lipschitz continuous over
	$\St(p,N)(\subset \mathbb{R}^{N\times p})$.
	Then
	\begin{equation}
		\textrm{find} \  \bm{U}^{\star} \in \argmin_{\bm{U}\in \St(p,N)}f(\bm{U}), \label{eq:cost}
	\end{equation}
	where the existence of such a minimizer
	$\bm{U}^{\star}$
	is guaranteed automatically by the compactness of
	$\St(p,N) (\subset \mathbb{R}^{N\times p})$
	and the continuity of
	$f$
	over the $Np$-dimensional Euclidean space
	$\mathbb{R}^{N\times p}$.
\end{problem}
Especially for
$p \ll N$,
Problem~\ref{problem:origin} often arises in data science, including signal processing and machine learning.
These applications include, e.g., nearest low-rank correlation matrix problem~\cite{Pietersz-Groenen04,Grubisic-Pietersz07,Zhu15}, nonlinear eigenvalue problem~\cite{Bai-Sleijpen-Vorst-Lippert-Edelman98,Yang-Meza-Wang06,Zhao-Bai-Jin15}, joint diagonalization problem for independent component analysis~\cite{Joho-Mathis02,Theis-Cason-Absil09,Sato17,Nikpour02}, orthogonal Procrustes problem~\cite{Elden-Park99,Francisco-Viloche-Weber2017,Zhao-Wang-Nie16,Zhang-Yang-Shen-Ying20} and enhancement of the generalization performance in deep neural network~\cite{Helfrich-Willmott-Ye18,Bansal-Chen-Wang18}.
However, Problem~\ref{problem:origin} has inherent difficulties regarding the severe nonlinearity of
$\St(p,N)$.

{\it A Cayley parametrization (CP) strategy}~\cite{Yamada-Ezaki03,Fraikin-Huper-Dooren07,Hori-Tanaka10,Helfrich-Willmott-Ye18,Kume-Yamada19,Kume-Yamada20,Kume-Yamada21,Kume-Yamada22} resolves the nonlinearity of
$\St(p,N)$
in Problem~\ref{problem:origin} in terms of a Euclidean space with {\it a Cayley transform}.
The classical Cayley transform is defined for parametrization of the special orthogonal group
${\rm SO}(N):= \{\bm{U} \in {\rm O}(N):=\St(N,N) \mid \det(\bm{U}) = 1\}$
as
\begin{equation}
	\varphi:{\rm SO}(N) \setminus E_{N,N} \to Q_{N,N}:\bm{U} \mapsto (\bm{I}-\bm{U})(\bm{I}+\bm{U})^{-1}, \label{eq:CT_O}
\end{equation}
and its inversion is given by
\begin{equation}
	\varphi^{-1}:Q_{N,N}\to {\rm SO}(N) \setminus E_{N,N}:\bm{V}\mapsto (\bm{I}-\bm{V})(\bm{I}+\bm{V})^{-1}, \label{eq:ICT_O}
\end{equation}
where
$E_{N,N} := \{\bm{U} \in {\rm O}(N) \mid \det(\bm{I} + \bm{U}) = 0\}$
is called the singular-point set of
$\varphi$,
and
$Q_{N,N}:= \{\bm{V} \in \mathbb{R}^{N\times N} \mid \bm{V}^{\TT} = -\bm{V}\}$
is the set of all skew-symmetric matrices.
Since
$Q_{N,N}$
is clearly a vector space over
$\mathbb{R}$,
and
${\rm SO}(N) \setminus E_{N,N}$
is an open dense subset of
${\rm SO}(N)$~\cite{Yamada-Ezaki03},
this dense subset can be parameterized in terms of the vector space
$Q_{N,N}$
with
$\varphi^{-1}$.
For
$\St(p,N)$
with
$p<N$,
the Cayley transform-pair
$(\varphi,\varphi^{-1})$
has been extended to
$\Phi_{\bm{S}}:\St(p,N)\setminus E_{N,p}(\bm{S})\to Q_{N,p}(\bm{S})$
and
$\Phi_{\bm{S}}^{-1}:Q_{N,p}(\bm{S}) \to \St(p,N)\setminus E_{N,p}(\bm{S})$
with a given
$\bm{S} \in {\rm O}(N)$~\cite{Kume-Yamada22},
where
$E_{N,p}(\bm{S}) \left(\subset \St(p,N)\right)$
is a generalization of
$E_{N,N}$
for
$\Phi_{\bm{S}}$
and
\begin{equation} \label{eq:skew}
	Q_{N,p}(\bm{S}) := \left\{
	\left.\begin{bmatrix} \bm{A} & -\bm{B}^{\TT} \\ \bm{B} & \bm{0} \end{bmatrix}
	\right\vert \; \substack{ -\bm{A}^{\TT} = \bm{A} \in \mathbb{R}^{p\times p},\\ \bm{B} \in \mathbb{R}^{(N-p)\times p}}\right\}=:Q_{N,p} \subset Q_{N,N}
\end{equation}
is an
$Np-\frac{1}{2}p(p+1)$-dimensional vector space of skew-symmetric matrices (see~\eqref{eq:Cayley} and~\eqref{eq:Cayley_inv_origin} in Section~\ref{sec:Cayley}).
Although, for each
$\bm{S} \in {\rm O}(N)$,
$Q_{N,p}(\bm{S})$
is nothing but the common set
$Q_{N,p}$,
we distinguish them as the domain of parametrization
$\Phi_{\bm{S}}^{-1}$
of the particular subset
$\St(p,N)\setminus E_{N,p}(\bm{S})$.

Since Problem~\ref{problem:origin} is a nonconvex optimization problem in general, a realistic goal for Problem~\ref{problem:origin} has been to find a stationary point
$\bm{U}^{\star} \in \St(p,N)$
of
$f$
over
$\St(p,N)$~\cite{Absil-Mahony-Sepulchre08,Wen-Yin13,Gao-Liu-Chen-Yuan18}.
The problem to find a stationary point of Problem~\ref{problem:origin} can be restated as 
\begin{problem} \label{problem:ALCP_grad}
	Let
	$f:\mathbb{R}^{N\times p }\to \mathbb{R}$,
	be a differentiable function, and let its gradient
	$\nabla f:\mathbb{R}^{N\times p}\to \mathbb{R}^{N\times p}$
	Lipschitz continuous over
	$\St(p,N)(\subset \mathbb{R}^{N\times p})$.
	Then
	\begin{equation}\label{eq:problem_alcp_grad}
		\textrm{find} \ (\bm{V}^{\star},\bm{S}^{\star}) \in  Q_{N,p}(\bm{S}^{\star}) \times {\rm O}(N) \ \textrm{such that} \  \|\nabla (f\circ\Phi_{\bm{S}^{\star}}^{-1})(\bm{V}^{\star})\|_{F} = 0,
	\end{equation}
	where
	$\nabla (f\circ\Phi_{\bm{S^{\star}}}^{-1})$
	is the gradient of
	$f\circ\Phi_{\bm{S^{\star}}}^{-1}$
	under the standard inner product
	$ \inprod{\bm{V}_{1}}{\bm{V}_{2}} := \trace(\bm{V}_{1}^{\TT}\bm{V}_{2})\ (\bm{V}_{1},\bm{V}_{2} \in Q_{N,p}(\bm{S^{\star}}))$
	(see Fact~\ref{fact:gradient} for an explicit formula of
	$\nabla (f\circ \Phi_{\bm{S^{\star}}}^{-1})$).
	The existence of such a solution
	$(\bm{V}^{\star},\bm{S}^{\star})$
	is guaranteed automatically because, for a minimizer
	$\bm{U}^{\star}$
	in Problem~\ref{problem:origin},
	(i)
	$\bm{U}^{\star}$
	is a stationary point of Problem~\ref{problem:origin} (see Remark~\ref{remark:stationary}~(a));
	(ii) there exists
	$\bm{S}^{\star} \in {\rm O}(N)$
	such that
	$\bm{U}^{\star} \in \St(p,N)\setminus E_{N,p}(\bm{S}^{\star})$
	(see Fact~\ref{fact:center_point});
	(iii)
	$\Phi_{\bm{S}^{\star}}(\bm{U}^{\star}) \in Q_{N,p}(\bm{S}^{\star})$
	is a stationary point of
	$f\circ\Phi_{\bm{S}^{\star}}^{-1}$ (see Remark~\ref{remark:stationary}~(b)).
\end{problem}

\begin{remark}[Stationary point of Problem~\ref{problem:origin}] \label{remark:stationary}
	\mbox{}
	\begin{enumerate}[label=(\alph*)]
		\item
			$\bm{U}^{\star} \in \St(p,N)$
			is said to be a stationary point of Problem~\ref{problem:origin} (see, e.g.,~\cite[Definition~2.1, Remark~2.3]{Gao-Liu-Chen-Yuan18} and~\cite[Lemma~1]{Wen-Yin13}) if
			$\bm{U}^{\star}$
			satisfies
			\begin{equation}
				\begin{cases}
					(\bm{I}-\bm{U}^{\star}\bm{U}^{\star\TT})\nabla f(\bm{U}^{\star}) & = \bm{0} \\
					\bm{U}^{\star\TT}\nabla f(\bm{U}^{\star}) -\nabla f(\bm{U}^{\star})^{\TT}\bm{U}^{\star} & = \bm{0}.
				\end{cases}\label{eq:optimality}
			\end{equation}
			Every minimizer of Problem~\ref{problem:origin} is a stationary point of Problem~\ref{problem:origin}~\cite{Wen-Yin13,Gao-Liu-Chen-Yuan18}.
		\item
			Let
			$\bm{U}^{\star} \in \St(p,N)$
			and
			$\bm{S}^{\star} \in {\rm O}(N)$
			satisfy
			$\bm{U}^{\star} \in \St(p,N) \setminus E_{N,p}(\bm{S}^{\star})$.
			Then,
			$\bm{U}^{\star}$
			is a stationary point of Problem~\ref{problem:origin} if and only if
			$\Phi_{\bm{S}^{\star}}(\bm{U^{\star}}) \in Q_{N,p}(\bm{S}^{\star})$
			is a stationary point of
			$f\circ\Phi_{\bm{S^{\star}}}^{-1}$
			over
			$Q_{N,p}(\bm{S^{\star}})$~\cite{Kume-Yamada22},
			i.e.,
			$\nabla (f\circ \Phi_{\bm{S^{\star}}}^{-1})(\Phi_{\bm{S^{\star}}}(\bm{U^{\star}})) = \bm{0}$.
	\end{enumerate}
\end{remark}

To find an approximate solution of Problem~\ref{problem:ALCP_grad}, {\it a naive Cayley parametrization (CP) strategy}~\cite{Yamada-Ezaki03,Helfrich-Willmott-Ye18,Kume-Yamada22} considers the following restricted version of Problem~\ref{problem:ALCP_grad} by fixing some
$\bm{S} \in {\rm O}(N)$
as
\begin{problem} \label{problem:CP_St}
	Let
	$f:\mathbb{R}^{N\times p }\to \mathbb{R}$,
	be a differentiable function, and let its gradient
	$\nabla f:\mathbb{R}^{N\times p}\to \mathbb{R}^{N\times p}$
	Lipschitz continuous over
	$\St(p,N)(\subset \mathbb{R}^{N\times p})$.
	For arbitrarily chosen
	$\bm{S} \in {\rm O}(N)$
	and
	$\epsilon > 0$,
	then
	\begin{equation}
		\textrm{find} \ \bm{V}^{\diamond} \in  Q_{N,p}(\bm{S})  \ \textrm{such that} \  \|\nabla (f\circ\Phi_{\bm{S}}^{-1})(\bm{V}^{\diamond})\|_{F} < \epsilon,
	\end{equation}
	where the existence of such an approximate stationary point
	$\bm{V}^{\diamond}$
	is guaranteed for every
	$\epsilon > 0$
	because of
	$\inf_{\bm{V}\in Q_{N,p}(\bm{S})} \|\nabla (f\circ\Phi_{\bm{S}}^{-1})(\bm{V})\|_{F} = 0$
	due to the denseness of
	$\St(p,N)\setminus E_{N,p}(\bm{S})$
	in
	$\St(p,N)$~\cite{Kume-Yamada22}
	(Note: the existence of a stationary point
	$\bm{V}^{\star} \in Q_{N,p}(\bm{S})$
	of
	$f\circ\Phi_{\bm{S}}^{-1}$
	is not guaranteed in general because the stationary point
	$\bm{U}^{\star}$
	of Problem~\ref{problem:origin} may belong to
	$E_{N,p}(\bm{S})$
	for the chosen
	$\bm{S}$).
\end{problem}
Since Problem~\ref{problem:CP_St} is defined over Euclidean space
$Q_{N,p}(\bm{S})$,
many powerful Euclidean optimization algorithms can be employed for Problem~\ref{problem:CP_St}, e.g., the gradient descent method~\cite{Nocedal-Wright06}, the Newton method~\cite{Nocedal-Wright06}, the quasi-Newton method~\cite{Nocedal-Wright06,Li-Fukushima01}, the conjugate gradient method~\cite{Andrei20,Gilbert-Nocedal92,Al-baali85,Dai-Yuan99,Dai-etal00,Dai-Yuan01,Hager-Zhang05}, the three-term conjugate gradient method~\cite{Zhang-Zhou-Li07,Narushima-Yabe-Ford11,Khoshsimaye-Ashrafi23}, and the Nesterov-type accelerated gradient method~\cite{Nesterov83,Ghadimi-Lan16,Carmon-Duchi-Hinder-Sidford18,Zeyuan18,Diakonikolas-Jordan21}.

However, in the above {\it naive CP strategy}, there is a risk of performance degradation in a case where a solution
$\bm{U}^{\star}$
to Problem~\ref{problem:origin} is close to the singular-point set
$E_{N,p}(\bm{S})$~\cite{Yamada-Ezaki03,Kume-Yamada22}.
This performance degradation, called {\it a singular-point issue} in this paper, can be explained intuitively that every singular point corresponds to a point at infinity in
$Q_{N,p}(\bm{S})$
via
$\Phi_{\bm{S}}^{-1}$.
To explain more precisely, consider the case where an estimate
$\bm{V}_{n}\in Q_{N,p}(\bm{S})$
for a solution of Problem~\ref{problem:CP_St} is far away from zero.
Then, even if
$\bm{V}_{n}$
is greatly updated to
$\bm{V}_{n+1} \in Q_{N,p}(\bm{S})$,
the corresponding updating from
$\bm{U}_{n}:=\Phi_{\bm{S}}^{-1}(\bm{V}_{n}) \in \St(p,N)$
to
$\bm{U}_{n+1}:=\Phi_{\bm{S}}^{-1}(\bm{V}_{n+1}) \in \St(p,N)$
can become tiny (see just after Fact~\ref{fact:mobility}).
To avoid such a singular-point issue, we are desired to design a good
$\bm{S} \in {\rm O}(N)$
such that
$\Phi_{\bm{S}}(\bm{U}^{\star}) \in Q_{N,p}(\bm{S})$
is located not distant from zero,
which does not seem to be possible before solving Problem~\ref{problem:CP_St}.

In this paper, to mitigate the singular-point issue in the naive CP strategy, we consider the estimation of good
$\bm{S}^{\star} \in {\rm O}(N)$
in addition to
$\bm{V}^{\star}$
in Problem~\ref{problem:ALCP_grad}.
To solve Problem~\ref{problem:ALCP_grad}, we propose a modified version of the naive CP strategy, named
{\it an Adaptive Localized Cayley Parametrization (ALCP) strategy} in Algorithm~\ref{alg:ALCP}, with a scheme to change
$\bm{S} \in {\rm O}(N)$
adaptively in such a case where a risk of the singular-point issue is detected.
While keeping the same
$\bm{S} \in {\rm O}(N)$,
the ALCP strategy performs the same processing as the naive CP strategy.
Just after the detection of the risk of the singular-point issue, we change the center point
$\bm{S} \in {\rm O}(N)$
to
$\bm{S}' \in {\rm O}(N)$
so that a reparameterized estimate
$\bm{V}_{n+1}:= \Phi_{\bm{S}'}\circ\Phi_{\bm{S}}^{-1}(\bm{\widetilde{V}}_{n+1}) \in Q_{N,p}(\bm{S}')$
stays not distant from zero, where
$\bm{\widetilde{V}}_{n+1}\in Q_{N,p}(\bm{S})$
is the updated estimate of a stationary point of
$f\circ\Phi_{\bm{S}}^{-1}$.
Then, we restart to apply a Euclidean optimization algorithm to the new minimization of
$f\circ\Phi_{\bm{S}'}^{-1}$
with the initial point
$\bm{V}_{n+1}$.

In the ALCP strategy, we can employ, in principle, any Euclidean optimization algorithm for minimization of
$f\circ\Phi_{\bm{S}}^{-1}$
while keeping the same
$\bm{S} \in {\rm O}(N)$.
However, since the strategy considers the minimization of time-varying functions
$f\circ\Phi_{\bm{S}}^{-1}$
over time-varying Euclidean spaces
$Q_{N,p}(\bm{S})$
due to changing
$\bm{S}$,
we encounter the following question:
\begin{quote}
{\it Does
$(\bm{V}_{n})_{n=0}^{\infty}$
generated by the ALCP strategy have any convergence property?}
\end{quote}
To this natural question, in this paper, we present affirmatively a unified convergence analysis for the ALCP strategy incorporating a fairly standard class of Euclidean optimization algorithms.
These include, e.g., the gradient descent method~\cite{Nocedal-Wright06}, the conjugate gradient method~\cite{Andrei20,Gilbert-Nocedal92,Al-baali85,Dai-Yuan99,Dai-etal00,Dai-Yuan01,Hager-Zhang05}, the three-term conjugate gradient method~\cite{Zhang-Zhou-Li07,Narushima-Yabe-Ford11,Khoshsimaye-Ashrafi23}, and the quasi-Newton method~\cite{Li-Fukushima01}.
Since the proposed convergence analysis relies on certain common convergence properties ensured by this class of Euclidean optimization algorithms, we establish a unified convergence analysis regarding the gradient to zero for the ALCP strategy incorporating such a class of Euclidean optimization algorithms.
The numerical experiments demonstrate that the ALCP strategy incorporating such a class of Euclidean optimization algorithms\footnote{
	Numerical experiments in preliminary reports~\cite{Kume-Yamada19,Kume-Yamada20} show that the ALCP strategy outperforms the {\it retraction-based strategy} for optimization over
	$\St(p,N)$
	even in the cases where more elaborated Euclidean optimization algorithms (e.g., the so-called Anderson acceleration~\cite{Kume-Yamada19}, and the Nesterov-type accelerated gradient method~\cite{Kume-Yamada20}) are employed.
} outperforms the standard retraction-based strategy.

We remark that the proposed ALCP strategy can also be interpreted along another line of research, called {\it a dynamic trivialization}~\cite{Lezcano19,Criscitiello-Boumal19,Lezcano20,Criscitiello-Boumal22}, for Problem~\ref{problem:origin} with {\it a retraction}~\cite{Absil-Mahony-Sepulchre08} because
$\Phi_{\bm{S}}^{-1}$
can be seen as the Cayley transform-based retraction~\cite{Wen-Yin13} through a specially designed invertible linear operator~\cite{Kume-Yamada22} (see Section~\ref{sec:retraction}).
So far, theoretical analyses for the dynamic trivialization seem to have been limited to the case where a computationally inefficient retraction, called {\it the exponential mapping}, is used.
Exceptionally, the proposed convergence analysis can be seen as a theoretical analysis for the dynamic trivialization with the efficiently available Cayley transform-based retraction~\cite{Wen-Yin13}.

\textbf{Notation}
$\mathbb{N}$,
$\mathbb{N}_{0}$,
and
$\mathbb{R}$
denote respectively the set of all positive integers, the set of all nonnegative integers, and the set of all real numbers.
For general
$n\in \mathbb{N}$,
$\bm{I}_{n} \in \mathbb{R}^{n\times n}$
stands for the identity matrix in
$\mathbb{R}^{n\times n}$,
but the identity matrix in
$\mathbb{R}^{N\times N}$
is denoted simply by
$\bm{I} \in \mathbb{R}^{N\times N}$.
For
$p \leq N$,
$\bm{I}_{N\times p} \in \mathbb{R}^{N\times p}$
denotes the matrix of the first
$p$
columns of
$\bm{I}$.
For
$p< N$,
the matrices
$\bm{U}_{\rm up} \in \mathbb{R}^{p \times p}$
and
$\bm{U}_{\rm lo} \in \mathbb{R}^{(N-p) \times p}$
denote respectively the upper and the lower block matrices of
$\bm{U} \in \mathbb{R}^{N\times p}$,
i.e.,
$\bm{U}=[\bm{U}_{\rm up}^{\TT}\ \bm{U}_{\rm lo}^{\TT}]^{\TT}$.
The matrices
$\bm{S}_{\rm le} \in \mathbb{R}^{N\times p}$
and
$\bm{S}_{\rm ri} \in \mathbb{R}^{N\times (N-p)}$
denote respectively the left and right block matrices of
$\bm{S} \in \mathbb{R}^{N\times N}$,
i.e.,
$\bm{S}=[\bm{S}_{\rm le}\ \bm{S}_{\rm ri}]$.
For a square matrix
$\bm{X}:= \begin{bmatrix} \bm{X}_{11} \in \mathbb{R}^{p\times p} & \bm{X}_{12} \in \mathbb{R}^{p\times (N-p)} \\ \bm{X}_{21} \in \mathbb{R}^{(N-p)\times p} & \bm{X}_{22} \in \mathbb{R}^{(N-p)\times (N-p)} \end{bmatrix}\in \mathbb{R}^{N\times N}$,
we use the notation
$\dbra{\bm{X}}_{ij}:= \bm{X}_{ij}$
for
$i,j \in \{1,2\}$.
For
$\bm{X} \in \mathbb{R}^{m\times n}$,
$\bm{X}^{\TT}$
denotes the transpose of
$\bm{X}$.
In particular for
$\bm{X}\in\mathbb{R}^{n\times n}$,
$\Skew(\bm{X}) = (\bm{X}-\bm{X}^{\TT})/2$
is the skew-symmetric component of
$\bm{X}$.
For square matrices
$\bm{X}_{i}\in \mathbb{R}^{n_{i}\times n_{i}}\ (1\leq i \leq k)$,
$\diag(\bm{X}_{1},\bm{X}_{2},\ldots,\bm{X}_{k})\in \mathbb{R}^{(\sum_{i=1}^{k}n_{i})\times  (\sum_{i=1}^{k}n_{i})}$
denotes the block diagonal matrix with diagonal blocks
$\bm{X}_{1},\bm{X}_{2},\ldots,\bm{X}_{k}$.
For a given matrix
$\bm{X} \in \mathbb{R}^{m\times n}$,
(i)
$\|\bm{X}\|_{2}$
and
$\|\bm{X}\|_{F}$
denote respectively the spectral norm and the Frobenius norm,
(ii)
$\sigma_{\max}(\bm{X})$
and
$\sigma_{\min}(\bm{X})$
denote respectively the nonnegative largest and the smallest singular values of
$\bm{X}$.
For a subset
$\mathcal{N}$
of
$\mathbb{N}_{0}$,
$\abs{\mathcal{N}}$
denotes the cardinality of
$\mathcal{N}$.
To distinguish from the symbol used for the orthogonal group
${\rm O}(N)$,
the symbol
$\mathfrak{o}(\cdot)$
will be used in place of the standard big O notation.
For a differentiable mapping
$F:\mathcal{X}\to\mathcal{Y}$
between Euclidean spaces
$\mathcal{X}$
and
$\mathcal{Y}$,
its G\^{a}teaux derivative at
$\bm{x}\in \mathcal{X}$
is the linear mapping
$\mathrm{D}F(\bm{x}):\mathcal{X} \to \mathcal{Y}$
defined, with a real variable
$t(\neq 0)$,
by
\begin{equation}
	(\bm{v}\in \mathcal{X}) \quad \mathrm{D}F(\bm{x})[\bm{v}] = \lim_{t\to 0}\frac{F(\bm{x}+t\bm{v}) - F(\bm{x})}{t}.
\end{equation}
For a differentiable function
$J:\mathcal{X} \to \mathbb{R}$
defined over the Euclidean space equipped with an inner product
$\inprod{\cdot}{\cdot}$,
$\nabla J(\bm{x}) \in \mathcal{X}$
is the gradient of
$J$
at
$\bm{x} \in \mathcal{X}$,
i.e.,
$\mathrm{D}J(\bm{x})[\bm{v}] = \inprod{\nabla J(\bm{x})}{\bm{v}}$
for all
$\bm{v} \in \mathcal{X}$.

\section{Preliminaries}\label{sec:preliminary}
\subsection{Generalized Cayley transform for $\St(p,N)$}\label{sec:Cayley}
For
$p,N\in \mathbb{N}$
satisfying
$p\leq N$,
the generalized Cayley transform
$\Phi_{\bm{S}}$
with
$\bm{S}\in {\rm O}(N)$
for an open dense parametrization of
$\St(p,N)$
is defined as
\begin{align}
	\Phi_{\bm{S}}:\St(p,N) \setminus E_{N,p}(\bm{S})\to Q_{N,p}(\bm{S}):\bm{U}\mapsto
	\begin{bmatrix}
		\bm{A}_{\bm{S}}(\bm{U}) & -\bm{B}^{\TT}_{\bm{S}}(\bm{U}) \\
		\bm{B}_{\bm{S}}(\bm{U}) & \bm{0}
	\end{bmatrix} 
	\label{eq:Cayley}
\end{align}
with
\begin{align}
	\bm{A}_{\bm{S}}(\bm{U}) &:= 2(\bm{I}_{p}+\bm{S}_{\rm le}^{\TT}\bm{U})^{-\mathrm{T}}\Skew(\bm{U}^{\TT}\bm{S}_{\rm le})(\bm{I}_{p}+\bm{S}_{\rm le}^{\TT}\bm{U})^{-1} \in Q_{p,p} \label{eq:Cay_A}\\
	\bm{B}_{\bm{S}}(\bm{U}) &:= - \bm{S}_{\rm ri}^{\TT}\bm{U}(\bm{I}_{p}+\bm{S}_{\rm le}^{\TT}\bm{U})^{-1} \in \mathbb{R}^{(N-p) \times p}, \label{eq:Cay_B}
\end{align}
where
\begin{equation}
	E_{N,p}(\bm{S}):= \{\bm{U} \in \St(p,N) \mid \det(\bm{I}_{p} + \bm{S}_{\rm le}^{\TT}\bm{U}) = 0\} \label{eq:singular}
\end{equation}
is called the singular-point set of
$\Phi_{\bm{S}}$
and
$Q_{N,p}(\bm{S})$
(see~\eqref{eq:skew}) is clearly a vector space~\cite{Kume-Yamada22}.
For every
$\bm{S} \in {\rm O}(N)$,
$\Phi_{\bm{S}}$
is a diffeomorphism between the special subset 
$\St(p,N)\setminus E_{N,p}(\bm{S})\left(\subset \St(p,N)\right)$
and the vector space
$Q_{N,p}(\bm{S})$~\cite{Kume-Yamada22}.
The inversion mapping of
$\Phi_{\bm{S}}$
is given, in terms of
$\varphi^{-1}$
in~\eqref{eq:ICT_O},
by
\begin{align}
	\Phi_{\bm{S}}^{-1}: Q_{N,p}(\bm{S})\to \St(p,N)\setminus E_{N,p}(\bm{S})
	:\bm{V} \mapsto
	& \bm{S}\varphi^{-1}(\bm{V})\bm{I}_{N\times p}=\bm{S}(\bm{I}-\bm{V})(\bm{I}+\bm{V})^{-1}\bm{I}_{N\times p} \!\!\!\!\!\!\\
	 & = 2(\bm{S}_{\rm le} - \bm{S}_{\rm ri}\dbra{\bm{V}}_{21})\bm{M}^{-1}-\bm{S}_{\rm le}, \label{eq:Cayley_inv_origin}
\end{align}
where
$\bm{M}:= \bm{I}_{p}+\dbra{\bm{V}}_{11}+\dbra{\bm{V}}_{21}^{\TT}\dbra{\bm{V}}_{21}\in \mathbb{R}^{p\times p}$
is the Schur complement matrix~\cite{Horn-Johonson12} of
$\bm{I}+\bm{V}\in \mathbb{R}^{N\times N}$.
We call
$\bm{S}$
{\it a center point} of
$\Phi_{\bm{S}}$
because its left block matrix
$\bm{S}_{\rm le} \in \St(p,N)\setminus E_{N,p}(\bm{S})$
corresponds to the origin
$\Phi_{\bm{S}}(\bm{S}_{\rm le}) = \bm{0}$
in the vector space
$Q_{N,p}(\bm{S})$.
We summarize useful properties of
$\Phi_{\bm{S}}$
as a parametrization of
$\St(p,N)$
in Fact~\ref{fact:property} below, implying respectively that~\ref{enum:dense}
$\St(p,N)$
can be parameterized almost globally by
$\Phi_{\bm{S}}$
with any
$\bm{S} \in {\rm O}(N)$;
\ref{enum:characterize_singular} every singular-point in
$E_{N,p}(\bm{S})$
is interpreted as a point at infinity in
$Q_{N,p}(\bm{S})$.
\begin{fact}[\cite{Kume-Yamada22}]\label{fact:property}
	Let
	$\bm{S} \in {\rm O}(N)$
	and
	$p < N$.
	Then, the following hold:
	\begin{enumerate}[label=(\alph*)]
		\item \label{enum:dense}
			$\St(p,N)\setminus E_{N,p}(\bm{S})$
			is an open dense\footnote{
				The closure of
				$\St(p,N)\setminus E_{N,p}(\bm{S})$
				is
				$\St(p,N)$.
				For every
				$\bm{U} \in \St(p,N)$,
				there exists a sequence
				$(\bm{U}_{n})_{n=0}^{\infty}( \subset \St(p,N) \setminus E_{N,p}(\bm{S}))$
				which converges to
				$\bm{U}$.
			} subset of
			$\St(p,N)$.
		\item \label{enum:characterize_singular}
			Let
			$g:Q_{N,p}(\bm{S})\to\mathbb{R}:\bm{V}\mapsto \det(\bm{I}_{p} + \bm{S}_{\rm le}^{\TT}\Phi_{\bm{S}}^{-1}(\bm{V}))$.
			Then,
			$\lim_{\|\bm{V}\|_{2}\to\infty} g(\bm{V}) = 0$.
			Conversely, if
			$(\bm{V}_{n})_{n=0}^{\infty} \subset Q_{N,p}(\bm{S})$
			satisfies
			$\lim_{n\to \infty}g(\bm{V}_{n}) = 0$,
			then
			$\lim_{n\to \infty}\|\bm{V}_{n}\|_{2} = \infty$.
	\end{enumerate}
\end{fact}

For every
$\bm{S} \in {\rm O}(N)$,
the computations of
$\Phi_{\bm{S}}$
and
$\Phi_{\bm{S}}^{-1}$
require
$\mathfrak{o}(N^{2}p+p^{3})$
flops (FLoating-point OPerationS [not 'FLoating point Operations Per Second']), which is high complexity especially for a small
$p$,
i.e., in the case of
$p\ll N$.
However, by employing a special center point
\begin{equation}
	\bm{S} \in \textrm{O}_{p}(N) := \left\{\diag(\bm{T},\bm{I}_{N-p})\mid  \bm{T}\in \textrm{O}(p)\right\}	 \subset {\rm O}(N), \label{eq:structure}
\end{equation}
these complexities for
$\Phi_{\bm{S}}$
and
$\Phi_{\bm{S}}^{-1}$
can be reduced to
$\mathfrak{o}(Np^{2} + p^{3})$
flops~\cite{Kume-Yamada22}.
For any
$\bm{U} \in \St(p,N)$,
Fact~\ref{fact:center_point} below ensures that Algorithm~\ref{alg:center_point} can design a center point
$\bm{S} \in {\rm O}_{p}(N)$
satisfying
$\bm{U} \in \St(p,N) \setminus E_{N,p}(\bm{S})$~\cite{Kume-Yamada22}.
In Section~\ref{sec:adaptive}, we will use Algorithm~\ref{alg:center_point} to choose a 'good' center point for applications of
$\Phi_{\bm{S}}^{-1}$
to Problem~\ref{problem:CP_St}.

\begin{fact}[Parametrization of $\St(p,N)$ by $\Phi_{\bm{S}}$ with $\bm{S} \in {\rm O}_{p}(N) \subset{\rm O}(N)$~\cite{Kume-Yamada22}]\label{fact:center_point}
	For any
	$\bm{U}\in \St(p,N)$,
	let
	$\bm{S} \in {\rm O}_{p}(N) (\subset {\rm O}(N))$
	be generated by Algorithm~\ref{alg:center_point}.
	Then, the following hold:
	\begin{enumerate}[label=(\alph*)]
		\item \label{enum:center_nonsingular}
			$\det(\bm{I}_{p}+\bm{S}_{\rm le}^{\TT}\bm{U}) \geq 1$
			and
			$\bm{U}\in \St(p,N)\setminus E_{N,p}(\bm{S})$.
		\item \label{enum:center_norm}
			$\bm{A}_{\bm{S}}(\bm{U}) \overset{\eqref{eq:Cay_A}}{=} \bm{0}$,
			$\|\bm{B}_{\bm{S}}(\bm{U})\|_{2} \overset{\eqref{eq:Cay_B}}{=} \|\bm{U}_{\rm lo}\bm{Q}_{2}(\bm{I}_{p}+\bm{\Sigma})^{-1}\bm{Q}_{2}^{\TT}\|_{2} \leq 1$
			and
			$\|\Phi_{\bm{S}}(\bm{U})\|_{2} \leq 1$.
	\end{enumerate}
\end{fact}

\begin{algorithm}[t] 
	\caption{Choice of center point }
	\label{alg:center_point}
	\begin{algorithmic}
		\Require
		$\bm{U} = \begin{bmatrix} \bm{U}_{\rm up}^{\TT} \in \mathbb{R}^{p\times p} & \bm{U}_{\rm lo}^{\TT} \in \mathbb{R}^{p\times (N-p)}  \end{bmatrix}^{\TT} \in \St(p,N)$
		\State
			Compute the singular value decomposition
			$\bm{U}_{\rm up} = \bm{Q}_{1}\bm{\Sigma}\bm{Q}_{2}^{\TT}$.
		\State
		\Comment{$\bm{Q}_{1},\bm{Q}_{2} \in {\rm O}(N)$, $\bm{\Sigma}\in \mathbb{R}^{p\times p}$ is a diagonal matrix  whose diagonal entries are\\\hspace{4em} singular values of $\bm{U}_{\rm up}$.}
		\State
		$\bm{S} \leftarrow \diag(\bm{Q}_{1}\bm{Q}_{2}^{\TT},\bm{I}_{N-p}) \in {\rm O}_{p}(N)$ in~\eqref{eq:structure}
		\Ensure
		$\bm{S} \in {\rm O}_{p}(N)(\subset {\rm O}(N))$
	\end{algorithmic}
\end{algorithm}

\subsection{Naive Cayley parametrization strategy} \label{sec:CP}
The naive Cayley parametrization strategy (CP) considers Problem~\ref{problem:CP_St}, i.e., the problem to find an approximate stationary point
$\bm{V}^{\diamond} \in Q_{N,p}(\bm{S})$
of
$f\circ\Phi_{\bm{S}}^{-1}$
over the Euclidean space
$Q_{N,p}(\bm{S})$
with a chosen fixed
$\bm{S} \in {\rm O}(N)$~\cite{Yamada-Ezaki03,Kume-Yamada22}
(see Remark~\ref{remark:stationary}~(b) for the relation between a stationary point of Problem~\ref{problem:origin} and that of
$f\circ\Phi_{\bm{S}}^{-1}$).
In the naive CP strategy, we can directly utilize powerful Euclidean optimization algorithms to generate estimates
$(\bm{V}_{n})_{n=0}^{\infty} \in Q_{N,p}(\bm{S})$
for a solution to Problem~\ref{problem:CP_St}.
However, the naive CP strategy may suffer from the slow convergence in a case where the current estimate
$\bm{V}_{n} \in Q_{N,p}(\bm{S})$
of a solution to Problem~\ref{problem:CP_St} is updated at a distant point from zero~\cite{Yamada-Ezaki03,Kume-Yamada22}.
This performance degradation is called {\it the singular-point issue} because
$\Phi_{\bm{S}}^{-1}(\bm{V})$
approaches a singular-point in
$E_{N,p}(\bm{S})$
as
$\|\bm{V}\|_{2} \to \infty$ (see Fact~\ref{fact:property}~\ref{enum:characterize_singular}).
A quantitative reason for the singular-point issue can be explained via {\it the mobility analysis of $\Phi_{\bm{S}}^{-1}$} below.
\begin{fact}[Mobility analysis of $\Phi_{\bm{S}}^{-1}$~\cite{Kume-Yamada22}] \label{fact:mobility}
	Let
	$p,N\in \mathbb{N}$
	satisfy
	$p < N$,
	and let
	$\bm{S} \in {\rm O}(N)$.
	Then, for
	$\tau > 0$
	and
	$\bm{\mathcal{E}}\in Q_{N,p}(\bm{S})$
	satisfying
	$\|\bm{\mathcal{E}}\|_{F} = 1$,
	we have
	\begin{equation}\label{eq:mobility}
	(\bm{V} \in Q_{N,p}(\bm{S})) \quad \|\Phi_{\bm{S}}^{-1}(\bm{V}+\tau \bm{\mathcal{E}}) - \Phi_{\bm{S}}^{-1}(\bm{V})\|_{F} \leq \tau r(\bm{V}),
	\end{equation}
	where
	\begin{align}
		r(\bm{V}):=
		\frac{2\sqrt{1+\|\dbra{\bm{V}}_{21}\|_{2}^{2}}}{1+\sigma_{\min}^{2}(\dbra{\bm{V}}_{21})}.\label{eq:mobility_rate}
	\end{align}
	We call
	$r:Q_{N,p}(\bm{S})\to \mathbb{R}$
	the {\it mobility} of
	$\Phi_{\bm{S}}^{-1}$,
	which is bounded below as
	\begin{equation}
		(\bm{V} \in Q_{N,p}(\bm{S})) \quad r(\bm{V}) \geq 2(1+\|\dbra{\bm{V}}_{21}\|_{2}^{2})^{-1/2}, \label{eq:bound_ratio}
	\end{equation}
	where the equality holds in~\eqref{eq:bound_ratio} when
	$\sigma_{\min}(\dbra{\bm{V}}_{21}) = \sigma_{\max}(\dbra{\bm{V}}_{21})(=\|\dbra{\bm{V}}_{21}\|_{2})$.
\end{fact}

The mobility
$r(\bm{V})$
of
$\Phi_{\bm{S}}^{-1}$
in Fact~\ref{fact:mobility} can serve as an indicator of the sensitivity of
$\Phi_{\bm{S}}^{-1}$
to the change at
$\bm{V}\in Q_{N,p}(\bm{S})$.
Since the small
$r(\bm{V})$
forces the change
$\|\Phi_{\bm{S}}^{-1}(\bm{V}+\tau \bm{\mathcal{E}}) - \Phi_{\bm{S}}^{-1}(\bm{V})\|_{F}$
to be small due to~\eqref{eq:mobility}, the slow convergence of the naive CP strategy likely occurs if estimates
$\bm{V}_{n} \in Q_{N,p}(\bm{S})$
are updated in the area where
$r(\bm{V}_{n})$
are small.

In the following, we consider two simple examples to see cases where the mobility can be small or large.
Under the equality condition of~\eqref{eq:bound_ratio}, i.e.,
$\sigma_{\min}(\dbra{\bm{V}}_{21}) = \sigma_{\max}(\dbra{\bm{V}}_{21})(=\|\dbra{\bm{V}}_{21}\|_{2})$,
the mobility
$r(\bm{V})$
becomes small when
$\|\dbra{\bm{V}}_{21}\|_{2}$
increases.
This example implies that
$r(\bm{V})$
tends to be small as
$\dbra{\bm{V}}_{21}$
increases, and thus there is a risk of the singular-point issue in a case where
$\bm{V}_{n}$
is updated at a distant point from zero.
On the other hand, such a risk is precluded if
$\bm{V}_{n}$
is updated not distant from zero because
$\bm{V} \in Q_{N,p}(\bm{S})$
around zero does not yield small
$r(\bm{V})$,
which is confirmed by the special example
$r(\bm{0}) = 2$.

The above mobility analysis suggests that updating
$\bm{V}_{n}$
not distant from zero is effective in avoiding the risk of the singular-point issue of the naive CP strategy.
To realize such updating, a center point
$\bm{S}$
is desired to be chosen strategically in order to make
$\Phi_{\bm{S}}(\bm{U}^{\star})$
not distant from zero for a global minimizer
$\bm{U}^{\star} \in \St(p,N)$
of Problem~\ref{problem:origin}.
However, finding such
$\bm{S}$
in advance is not realistic in general because even a good estimate of
$\bm{U}^{\star}$
is unknown before running optimization algorithms.

\section{Optimization over the Stiefel manifold with the adaptive localized Cayley parametrization}\label{sec:technique}
\subsection{Adaptive localized Cayley parametrization strategy}\label{sec:adaptive}
To circumvent the singular-point issue observed in the naive CP strategy, we present a Cayley parametrization strategy with an adaptive change of center points, named an {\it Adaptive Localized Cayley Parametrization (ALCP)} strategy in Algorithm~\ref{alg:ALCP}, where Table~\ref{table:notation} illustrates notations used in Algorithm~\ref{alg:ALCP}.
To use
$\Phi_{\bm{S}}$
and
$\Phi_{\bm{S}}^{-1}$
computationally efficiently within Algorithm~\ref{alg:ALCP}, we employ center points in
${\rm O}_{p}(N)$
(see~\eqref{eq:structure}) obtained by Algorithm~\ref{alg:center_point}.
Figure~\ref{fig:ALCP} is an illustration of the process in Algorithm~\ref{alg:ALCP}.
\begin{algorithm}[t]
	\caption{Adaptive localized Cayley parametrization strategy}
	\label{alg:ALCP}
	\begin{algorithmic}[1]
		\Require
		$\bm{U}_0 \in \St(p,N)$
		\State
		$l \leftarrow 0$,
		$n \leftarrow 0$,
		$\nu \leftarrow 0$
		\State
		$\bm{S}_{[l]} \leftarrow {\rm Algorithm~\ref{alg:center_point}}(\bm{U}_{n})$
		\State
		$\bm{V}_{n} \leftarrow \Phi_{\bm{S}_{[l]}}(\bm{U}_{n})$
		\State
		Initialize
		$\mathfrak{R}_{[\nu,n]}$
		(depend on the chosen
		$\mathcal{A}^{\langle l \rangle}$)
		\While{}
			\If{$\|\nabla (f\circ \Phi_{\bm{S}_{[l]}}^{-1}) (\bm{V}_{n})\|_{F} = 0$ or stopping criteria hold true}
				\State
				$\mathcal{N}_{l} \leftarrow \{k \in \mathbb{N}_{0} \mid \nu \leq k \leq n\}$
				\State
				{\bf break}
			\EndIf
			\State \label{lst:line:V}
			$\bm{\widetilde{V}}_{n+1} \leftarrow \mathcal{A}^{\langle l \rangle}(\bm{V}_{n},\mathfrak{R}_{[\nu,n]})$
			\State
			$\bm{U}_{n+1} \leftarrow \Phi_{\bm{S}_{[l]}}^{-1}(\bm{\widetilde{V}}_{n+1})$
			\If{Condition, e.g.,~\eqref{eq:change}, to detect of the singular-point issue holds} \label{lst:line:if}
			\State \label{lst:line:change_start}
				$\mathcal{N}_{l} \leftarrow \{k \in \mathbb{N}_{0} \mid \nu \leq k \leq n\}$
				\State
				$\nu\leftarrow n+1$
				\State
				$\bm{S}_{[l+1]} \leftarrow {\rm Algorithm~\ref{alg:center_point}}(\bm{U}_{n+1})$ \label{lst:line:change2}
				\Comment{Change the center point}
				\State
				$\bm{V}_{n+1} \leftarrow \Phi_{\bm{S}_{[l+1]}}(\bm{U}_{n+1})$ \label{lst:line:reparameterized}
				\Comment{Reparameterize $\bm{U}_{n+1}$ using $\bm{S}_{[l+1]}$}
				\State
				Reinitialize
				$\mathfrak{R}_{[\nu,n+1]}=\mathfrak{R}_{[n+1,n+1]}$
				(depend on the chosen
				$\mathcal{A}^{\langle l+1 \rangle}$)
				\State \label{lst:line:change_end}
				$l\leftarrow l+1$
			\Else
				\State
				$\bm{V}_{n+1} \leftarrow \bm{\widetilde{V}}_{n+1}$ \label{lst:line:V_not}
			\State
				Update from
				$\mathfrak{R}_{[\nu,n]}$
				to
				$\mathfrak{R}_{[\nu,n+1]}$
				(depend on the chosen
				$\mathcal{A}^{\langle l \rangle}$)
			\EndIf \label{lst:line:ifend}
			\State
			$n \leftarrow n+1$
		\EndWhile
		\Ensure
		$\bm{U}_{n} \in \St(p,N)$
	\end{algorithmic}
\end{algorithm}

\afterpage{%
\begin{table}[t]
	\caption{Notation in Algorithm~\ref{alg:ALCP}}
	\begin{tabularx}{\textwidth}{>{\hsize=.4\hsize}X>{\hsize=.6\hsize}X}  \label{table:notation} \\
		\toprule
$\mathcal{N}_{l} \subset \mathbb{N}_{0}\ (l\in \mathbb{N}_{0})$ & Set of iteration numbers during the $l$th center point $\bm{S}_{[l]}$ is used as a center point\\
$\bm{S}_{[l]} \in {\rm O}_{p}(N)\ (l\in \mathbb{N}_{0})$ & $l$th center point \\
$\bm{V}_{n} \in Q_{N,p}(\bm{S}_{[l]})\ (l\in \mathbb{N}_{0},n\in \mathcal{N}_{l})$ & $n$th estimate of an approximate solution of $f\circ\Phi_{\bm{S}_{[l]}}^{-1}$ over $Q_{N,p}(\bm{S}_{[l]})$ \\
$\bm{U}_{n} \in \St(p,N)\ (n\in \mathbb{N}_{0})$ & $n$th estimate of a solution to Problem~\ref{problem:origin} \\
$\mathfrak{R}_{[\nu,n]}\ (\nu,n \in \mathbb{N}_{0} \ \mathrm{s.t.} \ \nu \leq n)$ &
Strategic information at the update from $\bm{V}_{n}$ to $\bm{\widetilde{V}}_{n+1}$ obtained in the process of estimating
	$(\bm{V}_{k})_{k=\nu}^{n}$ (see Remark~\ref{remark:Euclidean_optimization_algorithm}) \\
$\mathcal{A}^{\langle l \rangle}\ (l\in \mathbb{N}_{0})$ &
Euclidean optimization algorithm for estimating an approximate stationary point of $f\circ\Phi_{\bm{S}_{[l]}}^{-1}$ over $Q_{N,p}(\bm{S}_{[l]})$
(see Remark~\ref{remark:Euclidean_optimization_algorithm}) \\
\bottomrule
\end{tabularx}
\end{table}
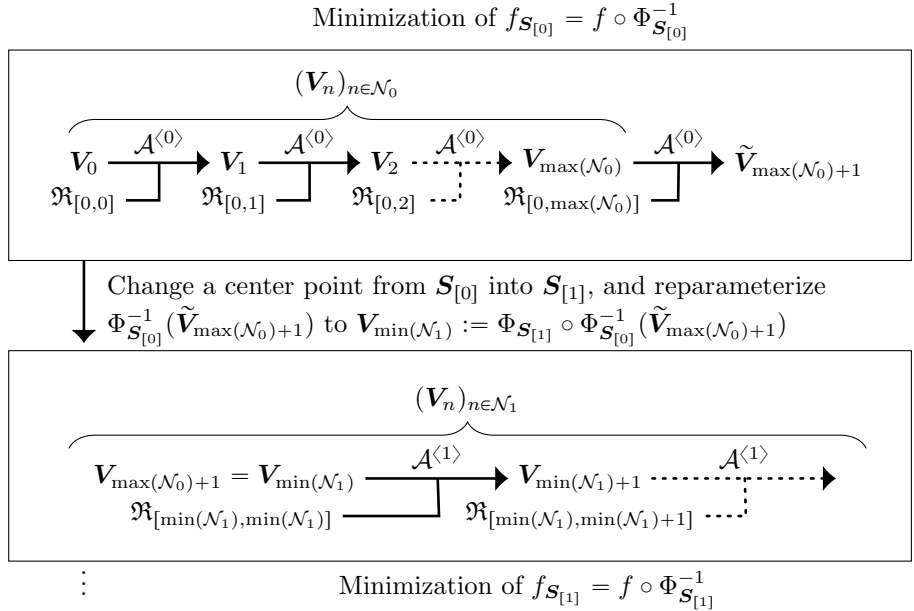
\begin{figure}[h]
	\centering
	\begin{tikzpicture}
		\pgfmathsetmacro\length{2}
		\pgfmathsetmacro\diffv{0.5}

		\draw (-1,-1.3) rectangle (5.5*\length,1.5) node[above,xshift=-155pt]{Minimization of $f_{\bm{S}_{[0]}}=f\circ\Phi_{\bm{S}_{[0]}}^{-1}$};

		\foreach \i in {0,1,2}{
			\draw (\length*\i,0)node(\i){$\bm{V}_{\i}$};
			\draw (\length*\i,-\diffv) node (\i_R) {$\mathfrak{R}_{[0,\i]}$};
		}
		\foreach \i  in {0,1}{
			\pgfmathsetmacro\j{int(\i+1)}
			\draw[line width=1pt, arrows = {Latex[width=10pt, length=5pt]-}] (\j) -- (\i) node [midway, above, sloped] (\i_arrow) {$\mathcal{A}^{\langle 0 \rangle}$};
			\draw[line width=1pt]  (\i_arrow) -- ($(\i_R)+(1,0)$) -- (\i_R);
		}
		\draw (\length*3+0.5,0)node(n){$\bm{V}_{\max(\mathcal{N}_{0})}$};
		\draw (\length*3+0.5,-\diffv) node (n_R) {$\mathfrak{R}_{[0,\max(\mathcal{N}_{0})]}$};
		\draw[line width=1pt, dash pattern=on 1pt off 3pt, line cap=round,arrows = {Latex[width=10pt, length=5pt]-}] (n) -- (2) node [midway, above, sloped] (n_arrow) {$\mathcal{A}^{\langle 0 \rangle}$};
		\draw[line width=1pt, dash pattern=on 1pt off 3pt, line cap=round]  (n_arrow) -- ($(2_R)+(1,0)$) -- (2_R);

		\draw (\length*4+1.5,0) node(n+1){$\bm{\widetilde{V}}_{\max(\mathcal{N}_{0})+1}$};
		\draw[line width=1pt, arrows = {Latex[width=10pt, length=5pt]-}] (n+1) -- (n) node [midway, above, sloped] (n+1_arrow) {$\mathcal{A}^{\langle 0 \rangle}$};
		\draw[line width=1pt]  (n+1_arrow) -- ($(n_R)+(1.4,0)$) -- (n_R);

		\draw [decorate,decoration={brace,amplitude=10pt}] (-0.2,0.4)--(3.6*\length,0.4) node (brace)[black,midway,yshift=10pt,pos=0.5,above]{$(\bm{V}_{n})_{n\in\mathcal{N}_{0}}$};

		\pgfmathsetmacro\h{-4.2}
		\draw[line width=1pt, arrows = {Latex[width=10pt, length=5pt]-}] (0,-2.4) -- (0, -1.3) node [pos=0,above right, align=left, yshift=-5pt, xshift=5pt] (TextNode) {Change a center point from $\bm{S}_{[0]}$ into $\bm{S}_{[1]}$, and reparameterize\\ $\Phi_{\bm{S}_{[0]}}^{-1}(\bm{\widetilde{V}}_{\max(\mathcal{N}_{0})+1})$ to $\bm{V}_{\min(\mathcal{N}_{1})}:=\Phi_{\bm{S}_{[1]}}\circ\Phi_{\bm{S}_{[0]}}^{-1}(\bm{\widetilde{V}}_{\max(\mathcal{N}_{0})+1})$};

		\draw (-1,-5.3) rectangle (5.5*\length,-2.5) node[below,xshift=-145pt,yshift=-80pt](N_1){Minimization of $f_{\bm{S}_{[1]}}=f\circ\Phi_{\bm{S}_{[1]}}^{-1}$};
		\draw (0,\h)node[right](new_n+1){$\bm{V}_{\max(\mathcal{N}_{0})+1} = \bm{V}_{\min(\mathcal{N}_{1})}$};
		\draw (0.5,\h-\diffv) node[right] (new_n+1_R) {$\mathfrak{R}_{[\min(\mathcal{N}_{1}),\min(\mathcal{N}_{1})]}$};

		\draw (3.3*\length,\h)node(n+2){$\bm{V}_{\min(\mathcal{N}_{1})+1}$};
		\draw (3.3*\length,\h-\diffv) node (n+2_R) {$\mathfrak{R}_{[\min(\mathcal{N}_{1}),\min(\mathcal{N}_{1})+1]}$};
		\draw[line width=1pt, arrows = {Latex[width=10pt, length=5pt]-}] (n+2) -- (new_n+1) node [midway, above, sloped] (n+1_arrow) {$\mathcal{A}^{\langle 1 \rangle}$};
		\draw[line width=1pt]  (n+1_arrow) -- ($(new_n+1_R)+(2.75,0)$) -- (new_n+1_R);

		\draw[line width=1pt, dash pattern=on 1pt off 3pt, line cap=round,arrows = {Latex[width=10pt, length=5pt]-}] (5*\length,\h) -- (n+2) node [midway, above, sloped] (n+2_arrow) {$\mathcal{A}^{\langle 1 \rangle}$};
		\draw[line width=1pt, dash pattern=on 1pt off 3pt, line cap=round]  (n+2_arrow) -- ($(n+2_R)+(2.2,0)$) -- (n+2_R);
		\draw [decorate,decoration={brace,amplitude=10pt}] (-0.2,-3.8)--(5.2*\length,-3.8) node (brace)[black,midway,yshift=10pt,pos=0.5,above]{$(\bm{V}_{n})_{n\in\mathcal{N}_{1}}$};
		\draw (0,-5.1) node[below](dots){$\vdots$};
	\end{tikzpicture}
	\caption{An illustration of the process of the ALCP strategy.}
	\label{fig:ALCP}
\end{figure}
}

Based on the mobility analysis of
$\Phi_{\bm{S}}^{-1}$
in Fact~\ref{fact:mobility}, we found a risk of slow convergence for the naive CP strategy in a case where estimates
$\bm{V}_{n} \in Q_{N,p}(\bm{S})$
of a solution to Problem~\ref{problem:CP_St} are updated at a distant point from zero.
In contrast to the naive CP strategy where
a center point
$\bm{S} \in {\rm O}(N)$
is treated as a predetermined parameter, the ALCP strategy tries to estimate a good center point according to the mobility analysis essentially for Problem~\ref{problem:ALCP_grad}.
To this end, center points
$\bm{S}$
are changed adaptively for
$\bm{V}_{n+1}=\Phi_{\bm{S}}(\bm{U}_{n+1})$
to stay not distant from zero in the process of Algorithm~\ref{alg:ALCP}.

While the $l$th center point
$\bm{S}_{[l]}\in {\rm O}_{p}(N)$
is kept in use, Algorithm~\ref{alg:ALCP} updates $(\bm{V}_{n})_{n\in\mathcal{N}_{l}}$
by using a Euclidean optimization algorithm, say
$\mathcal{A}^{\langle l \rangle}$
(see Remark~\ref{remark:Euclidean_optimization_algorithm}),
in the common Euclidean space
$Q_{N,p}(\bm{S}_{[l]})$,
where
$\mathcal{N}_{l}$
is the set of updating indices at which
$\bm{S}_{[l]}$
is used as a center point, i.e.,
$\mathcal{N}_{l}$
is an interval subset of
$\mathbb{N}_{0}$
satisfying\footnote{
				Since
				$\mathbb{N}_{0}$
				is a well-ordered set, every its nonempty subset has the minimum element.
			}
\begin{align}
	& \bigcup_{l\in\mathbb{N}_{0}}\mathcal{N}_{l} = \mathbb{N}_{0}, \quad (l_{1}\neq l_{2})\quad \mathcal{N}_{l_{1}} \cap \mathcal{N}_{l_{2}} = \emptyset, \\
	& 0 < \abs{\mathcal{N}_{l}} < \infty \Rightarrow \max(\mathcal{N}_{l})+1 = \min(\mathcal{N}_{l+1}). \label{eq:N_l_consecutive}
\end{align}
In Algorithm~\ref{alg:ALCP}, remark that we can employ any Euclidean optimization algorithm as
$\mathcal{A}^{\langle l \rangle}$
for estimating an approximate stationary point of
$f\circ\Phi_{\bm{S}_{[l]}}^{-1}$
over
$Q_{N,p}(\bm{S}_{[l]})$
in a way exactly same as the naive CP strategy.
In principle, we update the estimates
$(\bm{V}_{n},\bm{S}_{[l]})_{n\in\mathcal{N}_{l}} \subset Q_{N,p}(\bm{S}_{[l]}) \times {\rm O}_{p}(N)$
of a solution
$(\bm{V}^{\star},\bm{S}^{\star}) \in Q_{N,p}(\bm{S}^{\star})\times {\rm O}(N)$
to Problem~\ref{problem:ALCP_grad} by using
$\mathcal{A}^{\langle l \rangle}$,
and can obtain
$(\bm{U}_{n})_{n\in \mathcal{N}_{l}}$
if necessary by applying
$\Phi_{\bm{S}_{[l]}}^{-1}$
to
$(\bm{V}_{n})_{n\in\mathcal{N}_{l}}$
in Algorithm~\ref{alg:ALCP}.

\begin{remark}[Euclidean optimization algorithm] \label{remark:Euclidean_optimization_algorithm}
	Let
	$J:\mathcal{X}\to \mathbb{R}$
	be a differentiable function over the Euclidean space
	$\mathcal{X}$.
	Let
	$\mathcal{N} \subset \mathbb{N}_{0}$
	be an interval subset, and
	$\bm{x}_{\min(\mathcal{N})} \in \mathcal{X}$
	a given initial point for estimating a stationary point
	$\bm{x}^{\star} \in \mathcal{X}$
	of
	$J$.
	For all Euclidean optimization algorithms,
	each update from
	$\bm{x}_{n} \in \mathcal{X}$
	to
	$\bm{x}_{n+1} \in \mathcal{X}$
	for searching
	$\bm{x}^{\star} \in \mathcal{X}$
	can be expressed as
	\begin{equation} 
		(n\in \mathcal{N}) \quad \mathcal{A}: \left(\bm{x}_{n},\mathfrak{R}_{[\min(\mathcal{N}),n]}\right) \mapsto \bm{x}_{n+1} \label{eq:Euclidean_optimization}
	\end{equation}
	with certain strategic information
	$\mathfrak{R}_{[\min(\mathcal{N}),n]}$,
	e.g., a partial history of search directions.
	More precisely,
	$\mathfrak{R}_{[\min(\mathcal{N}),n]}$
	is assumed to become available in the process of estimating
	$(\bm{x}_{k})_{k=\min(\mathcal{N})}^{n}$,
	and depends on the chosen algorithm
	(see~\eqref{eq:record} in Section~\ref{sec:numerical_retraction} in the cases of the conjugate gradient method).
\end{remark}

In line~\ref{lst:line:if}, by using certain alarming conditions (see, e.g., Example~\ref{ex:condition} below),
Algorithm~\ref{alg:ALCP} detects if there is the risk of the singular-point issue around
$\bm{\widetilde{V}}_{n+1} \in Q_{N,p}(\bm{S}_{[l]})$
obtained in line~\ref{lst:line:V}.
In a case where the alarming condition holds true at
$\bm{\widetilde{V}}_{n+1}$,
the center point is changed from
$\bm{S}_{[l]} \in {\rm O}_{p}(N)$
to
$\bm{S}_{[l+1]} \in {\rm O}_{p}(N)$
by applying Algorithm~\ref{alg:center_point} to
$\bm{U}_{n+1}:= \Phi_{\bm{S}_{[l]}}^{-1}(\bm{\widetilde{V}}_{n+1}) \in \St(p,N)\setminus E_{N,p}(\bm{S}_{[l]})$.
After then,
$\bm{U}_{n+1}$
is reparameterized into
$\bm{V}_{n+1} := \Phi_{\bm{S}_{[l+1]}}^{-1}(\bm{U}_{n+1}) \in Q_{N,p}(\bm{S}_{[l+1]})$
with the new
$\bm{S}_{[l+1]}$
in line~\ref{lst:line:reparameterized}.
Since the reparameterized
$\bm{V}_{n+1}$
is guaranteed to satisfy
$\|\bm{V}_{n+1}\|_{2} \leq 1$
by Fact~\ref{fact:center_point}, the risk of the singular-point issue is automatically precluded according to the mobility analysis.

\begin{example}[Alarming conditions to detect the singular-point issue]\label{ex:condition}
	The mobility analysis in Fact~\ref{fact:mobility} suggests that the condition
	$\|\bm{\widetilde{V}}_{n+1}\|_{2} > T$
	with a predetermined
	$T > 0$
	can serve as a simple alarming condition for detection of the singular-point issue in line~\ref{lst:line:if} of Algorithm~\ref{alg:ALCP}.
	In view of the computational complexity, the exact computation of
	$\|\bm{\widetilde{V}}_{n+1}\|_{2}$
	requires
	$\mathfrak{o}(N^{3})$
	flops, which is certainly prohibited in particular for real-time applications to the case
	$p \ll N$.
	Instead of
	$\|\bm{\widetilde{V}}_{n+1}\|_{2} > T$,
	we can use the following surrogate alarming condition
	\begin{equation} \label{eq:change}
		\|\dbra{\bm{\widetilde{V}}_{n+1}}_{11}\|_{2} + \|\dbra{\bm{\widetilde{V}}_{n+1}}_{21}\|_{2} > T
	\end{equation}
	with the block matrices of
	$\bm{\widetilde{V}}_{n+1}$.
	In a case where the alarming condition~\eqref{eq:change} does not hold,
	$\|\bm{\widetilde{V}}_{n+1}\|_{2} \leq T$
	is guaranteed
	from the triangle inequality
	\begin{align}
		\|\bm{\widetilde{V}}_{n+1}\|_{2} & = 
		\begin{Vmatrix} \dbra{\bm{\widetilde{V}}_{n+1}}_{11} & -\dbra{\bm{\widetilde{V}}_{n+1}}_{21}^{\TT} \\ \dbra{\bm{\widetilde{V}}_{n+1}}_{21} & \bm{0} \end{Vmatrix}_{2}
		\leq \begin{Vmatrix} \dbra{\bm{\widetilde{V}}_{n+1}}_{11} & \bm{0} \\ \bm{0} & \bm{0} \end{Vmatrix}_{2}
		+ \begin{Vmatrix} \bm{0} & - \dbra{\bm{\widetilde{V}}_{n+1}}_{21}^{\TT} \\ \dbra{\bm{\widetilde{V}}_{n+1}}_{21} &\bm{0} \end{Vmatrix}_{2} \\
														 & = \|\dbra{\bm{\widetilde{V}}_{n+1}}_{11}\|_{2} + \|\dbra{\bm{\widetilde{V}}_{n+1}}_{21}\|_{2}
														 \leq T, \label{eq:triangle}
	\end{align}
	where the last equality is verified simply via an eigenvalue expression of the spectral norm
	$\|\cdot\|_{2}$\footnote{
	\begin{align}
		& \begin{Vmatrix} \bm{0} & - \dbra{\bm{\widetilde{V}}_{n+1}}_{21}^{\TT} \\ \dbra{\bm{\widetilde{V}}_{n+1}}_{21} &\bm{0} \end{Vmatrix}_{2}
		= \sqrt{\lambda_{\max}\left(\begin{bmatrix} \bm{0} & - \dbra{\bm{\widetilde{V}}_{n+1}}_{21}^{\TT} \\ \dbra{\bm{\widetilde{V}}_{n+1}}_{21} &\bm{0} \end{bmatrix}^{\TT}\begin{bmatrix} \bm{0} & - \dbra{\bm{\widetilde{V}}_{n+1}}_{21}^{\TT} \\ \dbra{\bm{\widetilde{V}}_{n+1}}_{21} &\bm{0} \end{bmatrix}\right)} \\
																											 & = \sqrt{\lambda_{\max}\left(\begin{bmatrix} \dbra{\bm{\widetilde{V}}_{n+1}}_{21}^{\TT}\dbra{\bm{\widetilde{V}}_{n+1}}_{21}&\bm{0} \\  \bm{0} & \dbra{\bm{\widetilde{V}}_{n+1}}_{21}\dbra{\bm{\widetilde{V}}_{n+1}}_{21}^{\TT}\end{bmatrix}\right)} \\
		& = \sqrt{\max\{\lambda_{\max}(\dbra{\bm{\widetilde{V}}_{n+1}}_{21}^{\TT}\dbra{\bm{\widetilde{V}}_{n+1}}_{21}),\lambda_{\max}(\dbra{\bm{\widetilde{V}}_{n+1}}_{21}\dbra{\bm{\widetilde{V}}_{n+1}}_{21}^{\TT})\}} \\
		& = \sqrt{\lambda_{\max}(\dbra{\bm{\widetilde{V}}_{n+1}}_{21}^{\TT}\dbra{\bm{\widetilde{V}}_{n+1}}_{21})}
		= \sqrt{\sigma_{\max}^{2}(\dbra{\bm{\widetilde{V}}_{n+1}}_{21})}  = \|\dbra{\bm{\widetilde{V}}_{n+1}}_{21}\|_{2}.
	\end{align}
}.
	The threshold
	$T$
	in~\eqref{eq:change}
	should be chosen for Algorithm~\ref{alg:ALCP} to enjoy efficacies of Euclidean optimization algorithms incorporated in Algorithm~\ref{alg:ALCP} because the center points could be changed too often, e.g., if
	$T$
	is set too small.
	To avoid such a situation, we recommend to employ the alarming condition~\eqref{eq:change} with
	$T>1$
	in line~\ref{lst:line:if} because
	$\|\dbra{\bm{V}_{n+1}}_{11}\|_{2} + \|\dbra{\bm{V}_{n+1}}_{21}\|_{2} \leq 1$
	holds for the reparameterized
	$\bm{V}_{n+1}$
	in line~\ref{lst:line:reparameterized}
	(see Fact~\ref{fact:center_point}~\ref{enum:center_norm}).
	We also remark that choices of conditions used in line~\ref{lst:line:if} are flexible as long as they can detect the singular-point issue.
\end{example}

By employing~\eqref{eq:change} as an alarming condition for detection of the singular-point issue in line~\ref{lst:line:if} of Algorithm~\ref{alg:ALCP}, we can ensure the boundedness of
$(\bm{V}_{n})_{n=0}^{\infty}$
below.
This boundedness will be used for our convergence analysis in Section~\ref{sec:convergence}.
\begin{lemma} \label{lemma:bounded}
	Let
	$f:\mathbb{R}^{N\times p} \to \mathbb{R}$
	be continuous, and
	let
	$(\bm{V}_{n})_{n=0}^{\infty}$
	be generated by Algorithm~\ref{alg:ALCP} in use of the condition~\eqref{eq:change} with
	$T > 0$
	in line~\ref{lst:line:if}.
	Then, it holds
	$\|\bm{V}_{n}\|_{2} \leq \max\{1,T\}\ (n\in\mathbb{N}_{0})$,
	i.e.,
	$(\bm{V}_{n})_{n=0}^{\infty}$
	is bounded.
\end{lemma}
\begin{proof}
	Let
	$l \in \mathbb{N}_{0}$.
	By the condition~\eqref{eq:change} in line~\ref{lst:line:if} and by line~\ref{lst:line:V_not} of Algorithm~\ref{alg:ALCP}, we have
	$\|\dbra{\bm{V}_{n+1}}_{11}\|_{2} + \|\dbra{\bm{V}_{n+1}}_{21}\|_{2} \leq T$
	as long as
	$n+1\in \mathcal{N}_{l}$.
	Combining this inequality and
	$\|\bm{V}_{n+1}\|_{2} \leq \|\dbra{\bm{V}_{n+1}}_{11}\|_{2} + \|\dbra{\bm{V}_{n+1}}_{21}\|_{2}$
	from~\eqref{eq:triangle},
	we have
	$\|\bm{V}_{n}\|_{2} \leq \max\{T,\|\bm{V}_{\min(\mathcal{N}_{l})}\|_{2}\}\ (n\in \mathcal{N}_{l})$.
	Since Fact~\ref{fact:center_point}~(b) ensures
	$\|\bm{V}_{\min(\mathcal{N}_{l})}\|_{2} \leq 1$,
	we obtain
	$\|\bm{V}_{n}\|_{2} \leq \max\{T,1\}\ (n\in \mathcal{N}_{l})$.
\end{proof}

\subsection{Convergence analysis for the ALCP strategy incorporating line-search methods of Armijo-type} \label{sec:convergence}
In this section, as fairly standard Euclidean optimization algorithms to be incorporated as
$\mathcal{A}^{\langle l \rangle}$
in Algorithm~\ref{alg:ALCP},
we consider line-search methods of Armijo-type, called in this paper Type A algorithm.
\begin{definition}[Type A algorithm: Line-search method of Armijo-type for interval $\mathcal{N} \subset \mathbb{N}_{0}$]\label{definition:line_search}
	Assume that a function
	$J:\mathcal{X}\to \mathbb{R}$
	is differentiable and
	$\nabla J$
	is Lipschitz continuous over the Euclidean space
	$\mathcal{X}$.
	For an interval subset
	$\mathcal{N} \subset \mathbb{N}_{0}$
	and a given point
	$\bm{x}_{\min(\mathcal{N})} \in \mathcal{X}$,
	we say a Euclidean optimization algorithm (see Remark~\ref{remark:Euclidean_optimization_algorithm})
	\begin{equation} 
		(n\in \mathcal{N}) \quad \mathcal{A}: \left(\bm{x}_{n},\mathfrak{R}_{[\min(\mathcal{N}),n]}\right) \mapsto \bm{x}_{n+1}:= \bm{x}_{n}+\gamma_{n}\bm{d}_{n} \in \mathcal{X}
	\end{equation}
	for estimating
	$\bm{x}^{\star} \in \mathcal{X}$
	s.t.
	$\nabla J(\bm{x}^{\star}) = \bm{0}$
	is a {\it Type A algorithm (line-search method of Armijo-type)}
	if
	a stepsize
	$\gamma_{n} > 0$
	and a search direction
	$\bm{d}_{n} \in \mathcal{X}$,
	determined by
	$\bm{x}_{n}$
	and
	$\mathfrak{R}_{[\min(\mathcal{N}),n]}$,
	satisfy
	\begin{enumerate}[label=(\alph*),leftmargin=*,align=left]
		\item
			$\bm{d}_{\min(\mathcal{N})} = -\nabla J(\bm{x}_{\min(\mathcal{N})})$
		\item
			If
			$\bm{x}_{n}$
			does not satisfy
			$\nabla J(\bm{x}_{n})\neq \bm{0}$,
			$\bm{d}_{n}$
			must satisfy the {\it descent condition} (see, e.g.,~\cite{Nocedal-Wright06}), i.e.,
			\begin{equation}
				(n\in\mathcal{N}) \quad \inprod{\nabla J(\bm{x}_{n})}{\bm{d}_{n}} < 0.
			\end{equation}
		\item
			$\gamma_{n}$
			must satisfy the {\it Armijo condition} (see, e.g.,~\cite{Nocedal-Wright06}) with some
			$c\in (0,1)$,
			i.e.,
			\begin{equation}
				(n\in\mathcal{N}) \quad J(\bm{x}_{n}+\gamma_{n}\bm{d}_{n}) \leq J(\bm{x}_{n}) + c\gamma_{n}\inprod{\nabla J(\bm{x}_{n})}{\bm{d}_{n}}. \label{eq:Armijo_Euclidean}
			\end{equation}
		\item
			Assume
			$\mathcal{N}$
			is infinite
			and
			$(\bm{x}_{n})_{n\in\mathcal{N}}$
			is bounded.
			Then,
			$\liminf_{n\to \infty,\\ n \in \mathcal{N}}\|\nabla J(\bm{x}_{n})\| = 0$
			is guaranteed.
	\end{enumerate}
\end{definition}

\begin{example}[Type A algorithm in Definition~\ref{definition:line_search}] \label{ex:Armijo}
	For minimization of
	$J:\mathcal{X}\to\mathcal{X}$
	over a Euclidean space
	$\mathcal{X}$,
	fairly standard Euclidean optimization algorithms  can be seen as special instances of Type A algorithms, e.g.,~the gradient descent method~\cite{Nocedal-Wright06}, the conjugate gradient method~\cite{Andrei20,Gilbert-Nocedal92,Al-baali85,Dai-Yuan99,Dai-etal00,Dai-Yuan01,Hager-Zhang05}, three-term conjugate gradient method~\cite{Zhang-Zhou-Li07,Narushima-Yabe-Ford11,Khoshsimaye-Ashrafi23}, and the quasi-Newton method~\cite{Li-Fukushima01}.
	For such fairly standard Euclidean optimization algorithms,
	the global convergence
	$\liminf_{n\to\infty} \|\nabla J(\bm{x}_{n})\| = 0$
	is guaranteed by assuming commonly
	(i) the boundedness of the level set
	$\mathrm{lev}_{\leq J(\bm{x}_{0})}J:=\{\bm{x}\in \mathcal{X} \mid J(\bm{x}) \leq J(\bm{x}_{0}) \}$,
	and
	(ii)
	the Lipschitz continuity of
	$\nabla J$
	over
	$\mathrm{lev}_{\leq J(\bm{x}_{0})}J$.
	On the other hand as remarked in~\cite[pp.97-98]{Andrei20},
	the 'boundedness assumption, i.e., (i), of the level set' is not necessarily always required because the boundedness of
	$\mathrm{lev}_{\leq J(\bm{x}_{0})}J$
	has been utilized just to ensure the boundedness of
	$(\bm{x}_{n})_{n=0}^{\infty}$,
	under the monotone decreasing of
	$(J(\bm{x}_{n}))_{n=0}^{\infty}$
	(which is guaranteed by Type A algorithms [see (a)-(c) in Definition~\ref{definition:line_search}]),
	in many convergence analyses for such fairly standard Euclidean optimization algorithms.
	In view of this observation for such algorithms, the global convergence
	$\liminf_{n\to\infty} \|\nabla J(\bm{x}_{n})\| = 0$
	is guaranteed even if the boundedness of
	$\mathrm{lev}_{\leq J(\bm{x}_{0})}J$
	is replaced by the boundedness of
	$(\bm{x}_{n})_{n=0}^{\infty}$
	as an assumption
	(see also Remark~\ref{remark:line-search} on Definition~\ref{definition:line_search_ALCP} for a reason why we will not assume the boundedness of the level set of
	$f_{\bm{S}}$).
\end{example}

In the following, we consider Algorithm~\ref{alg:ALCP} incorporating Type A algorithms as
$\mathcal{A}^{\langle l \rangle}\ (l \in \mathbb{N}_{0})$
in the sense of Definition~\ref{definition:line_search_ALCP} together with Remark~\ref{remark:conditions}.
For simplicity, we use notations
$f_{\bm{S}}:= f\circ\Phi_{\bm{S}}^{-1}$
and
$\nabla f_{\bm{S}}:=\nabla (f\circ\Phi_{\bm{S}}^{-1})$.
To identify the index
$l \in \mathbb{N}_{0}$
of the time-varying center point at $n$th update in Algorithm~\ref{alg:ALCP},
we introduce the following nondecreasing function satisfying
$\ell(n+1) - \ell(n) \in \{0,1\}$:
\begin{equation} \label{eq:ln}
	\ell:\mathbb{N}_{0} \to \mathbb{N}_{0}: n \mapsto l,\ \mathrm{s.t.},\ n\in \mathcal{N}_{l},\ \mathrm{i.e.},\ \bm{V}_{n} \in Q_{N,p}(\bm{S}_{[l]}).
\end{equation}
\begin{definition}[Algorithm~\ref{alg:ALCP} incorporating Type A algorithms] \label{definition:line_search_ALCP}
	Assume that a function
	$f:\mathbb{R}^{N\times p}\to \mathbb{R}$
	is differentiable and
	$\nabla f_{\bm{S}}$
	($\bm{S} \in {\rm O}(N)$)
	is Lipschitz continuous with a common Lipschitz constant
	$L>0$\footnote{
		This Lipschitz condition is satisfied if
		$\nabla f$
		is Lipschitz continuous over
		$\St(p,N) (\subset \mathbb{R}^{N\times p})$ (see Fact~\ref{fact:Lipschitz}).\label{foot:Lipschitz}
	}, i.e.,
	$f_{\bm{S}} = f\circ\Phi_{\bm{S}}^{-1}$
	satisfies
	\begin{equation} 
		(\exists L >0, \forall \bm{S} \in {\rm O}(N), \forall \bm{V}_{1},\bm{V}_{2} \in Q_{N,p}(\bm{S}))\
		\|\nabla f_{\bm{S}}(\bm{V}_{1}) - \nabla f_{\bm{S}}(\bm{V}_{2})\|_{F} \leq L\|\bm{V}_{1} - \bm{V}_{2}\|_{F}. \label{eq:Lipschitz}
	\end{equation}
	Then, we say that Algorithm~\ref{alg:ALCP} incorporates Type A algorithms if, for every
	$l \in \mathbb{N}_{0}$
	satisfying
	$\mathcal{N}_{l} \neq \emptyset$,
	$\mathcal{A}^{\langle l \rangle}$
	is a Type A algorithm on the interval subset
	$\mathcal{N}_{l} \subset \mathbb{N}_{0}$
	designed for estimating an approximate stationary point of
	$f_{\bm{S}_{[l]}}$
	over
	$Q_{N,p}(\bm{S}_{[l]})$,
	i.e., the following hold:
	\begin{enumerate}[label=(\alph*),leftmargin=*,align=left]
		\item
			For each
			$l \in \mathbb{N}_{0}$,
			$\bm{V}_{n}\in Q_{N,p}(\bm{S}_{[l]})\ (n\in\mathcal{N}_{l})$
			is updated to
			$\bm{\widetilde{V}}_{n+1} \in Q_{N,p}(\bm{S}_{[l]})$
			in line~\ref{lst:line:V} of Algorithm~\ref{alg:ALCP} as
			\begin{equation} \label{eq:update_rule}
				(n\in \mathcal{N}_{l}) \quad \bm{\widetilde{V}}_{n+1}:=\mathcal{A}^{\langle l \rangle}(\bm{V}_{n}, \mathfrak{R}_{[\min(\mathcal{N}_{l}),n]}) := \bm{V}_{n} + \gamma_{n}\bm{D}_{n},
			\end{equation}
			where
			$\gamma_{n} > 0$
			is a stepsize,
			$\bm{D}_{n} \in Q_{N,p}(\bm{S}_{[l]})$
			is a search direction, and
			$\bm{D}_{\min(\mathcal{N}_{l})} := - \nabla f_{\bm{S}_{[l]}}(\bm{V}_{\min(\mathcal{N}_{l})})$.
			Note that
			$\bm{V}_{n+1}$
			is determined in line~\ref{lst:line:reparameterized} or~\ref{lst:line:V_not} of Algorithm~\ref{alg:ALCP}.
		\item
			For each
			$l\in \mathbb{N}_{0}$
			and
			$n\in \mathcal{N}_{l}$
			satisfying
			$\nabla f_{\bm{S}_{[l]}}(\bm{V}_{n})\neq \bm{0}$,
			$\bm{D}_{n}$
			enjoys the descent condition, i.e.,
			\begin{equation} \label{eq:descent}
				(n\in\mathcal{N}_{l}) \quad \inprod{\nabla f_{\bm{S}_{[l]}}(\bm{V}_{n})}{\bm{D}_{n}} < 0.
			\end{equation}
		\item
			For each
			$l\in \mathbb{N}_{0}$,
			the stepsize
			$\gamma_{n}>0$
			and the search direction
			$\bm{D}_{n} \in Q_{N,p}(\bm{S}_{[l]})$
			satisfy the {\it Armijo condition} with some
			$c\in (0,1)$, i.e.,
			\begin{equation} \label{eq:Armijo}
				(n\in\mathcal{N}_{l}) \quad f_{\bm{S}_{[l]}}(\bm{V}_{n}+\gamma_{n}\bm{D}_{n}) \leq f_{\bm{S}_{[l]}}(\bm{V}_{n}) + c\gamma_{n}\inprod{\nabla f_{\bm{S}_{[l]}}(\bm{V}_{n})}{\bm{D}_{n}}.
			\end{equation}
		\item
			Assume
			$\widehat{l}:=\max_{n\in\mathbb{N}} \ell(n) \in \mathbb{N}_{0}$
			exists, implying thus
			$\mathcal{N}_{\widehat{l}} \subset \mathbb{N}_{0}$
			is a semi-infinite interval, i.e.,
			$[\min(\mathcal{N}_{\widehat{l}}),\infty) \cap \mathbb{N}_{0}$,
			where
			$\ell$
			is defined in~\eqref{eq:ln}.
			Assume
			$(\bm{V}_{n})_{n\in \mathcal{N}_{\widehat{l}}}$
			is bounded (which is guaranteed, e.g., if we employ~\eqref{eq:change} as the alarming condition in line~\ref{lst:line:if} of Algorithm~\ref{alg:ALCP}).
			Then, we can guarantee
			$\liminf_{n\to \infty, n \in \mathcal{N}_{\widehat{l}}}\|\nabla f_{\bm{S}_{[\widehat{l}]}}(\bm{V}_{n})\|_{F} = 0$.
	\end{enumerate}
\end{definition}

\begin{remark}[Conditions in Definition~\ref{definition:line_search_ALCP}]\label{remark:conditions}
	\mbox{}
	\begin{enumerate}[label=(\alph*)]
		\item
			The stepsize
			$\gamma_{n}$
			achieving the Armijo condition~\eqref{eq:Armijo} is available by {\it the backtracking algorithm} (see, e.g.,~\cite{Nocedal-Wright06}).
			The conditions (a)-(c) in Definition~\ref{definition:line_search_ALCP} guarantee the monotone decreasing of
			$(f_{\bm{S}_{[\ell(n)]}}(\bm{V}_{n}))_{n=0}^{\infty}$ (see, e.g.,~\cite[Figure 3.3]{Nocedal-Wright06}).
		\item
			The condition (d) in Definition~\ref{definition:line_search_ALCP} corresponds to the convergence property (Definition~\ref{definition:line_search} (d)) of the Type A algorithms to be employed in Algorithm~\ref{alg:ALCP}.
			The boundedness of
			$(\bm{x}_{n})_{n\in\mathcal{N}}$
			in Definition~\ref{definition:line_search} (d) is one of the key ingredients for convergence analyses, which has been used, e.g., to ensure the boundednesses of
			$(J(\bm{x}_{n}))_{n\in\mathcal{N}}$
			and
			$(\nabla J(\bm{x}_{n}))_{n\in\mathcal{N}}$.
			Fortunately for Algorithm~\ref{alg:ALCP}, the boundedness of
			$(\bm{V}_{n})_{n=0}^{\infty}$
			is guaranteed automatically if \eqref{eq:change} is employed as the alarming condition in line~\ref{lst:line:if} of Algorithm~\ref{alg:ALCP} (see Lemma~\ref{lemma:bounded}).
	\end{enumerate}
\end{remark}

\begin{remark}[Why the boundedness of $\mathrm{lev}_{\leq f_{\bm{S}_{[\widehat{l}]}}(\bm{V}_{\nu})}f_{\bm{S}_{[\widehat{l}]}}$ is not assumed in Definition~\ref{definition:line_search_ALCP} (d)?]\label{remark:line-search}
	In Definition~\ref{definition:line_search_ALCP} (d), under the existence of
	$\widehat{l}:=\max_{n\in\mathbb{N}}\ell(n) \in \mathbb{N}_{0}$
	with
	$\nu := \min(\mathcal{N}_{\widehat{l}})$,
	we do not assume the boundedness of
	$\mathrm{lev}_{\leq f_{\bm{S}_{[\widehat{l}]}}(\bm{V}_{\nu})}f_{\bm{S}_{[\widehat{l}]}}$
	unlike many existing convergence analyses for fairly standard Euclidean optimization algorithms (see Example~\ref{ex:Armijo}).
	This is because
	(i) the boundedness of
	$(\bm{V}_{n})_{n\in \mathcal{N}_{\widehat{l}}}$
	can be guaranteed by employing~\eqref{eq:change} as the alarming condition in line~\ref{lst:line:if} without assuming additionally the boundedness of
	$\mathrm{lev}_{\leq f_{\bm{S}_{[\widehat{l}]}}(\bm{V}_{\nu})}f_{\bm{S}_{[\widehat{l}]}}$
	(see Remark~\ref{remark:conditions}~(b)),
	and
	(ii) the boundedness of
	$\mathrm{lev}_{\leq f_{\bm{S}_{[\widehat{l}]}}(\bm{V}_{\nu})}f_{\bm{S}_{[\widehat{l}]}}$
	can not be guaranteed indeed.
	To explain (ii), consider the situation where a global minimizer
	$\bm{U}^{\star} \in \St(p,N)$
	of
	$f$
	satisfies
	$\bm{U}^{\star} \in E_{N,p}(\bm{S}_{[\widehat{l}]})$
	and
	$f(\bm{U}^{\star}) < f_{\bm{S}_{[\widehat{l}]}}(\bm{V}_{\nu})$\footnote{
		Since we can not check the satisfaction of these conditions before running the algorithm, this situation likely occurs in practice.
	}.
	Suppose
	$(\bm{V}_{n}^{\heartsuit})_{n=0}^{\infty} \subset Q_{N,p}(\bm{S}_{[\widehat{l}]})$
	achieves
	$\lim_{n\to\infty}\Phi_{\bm{S}_{[\widehat{l}]}}^{-1}(\bm{V}_{n}^{\heartsuit}) = \bm{U}^{\star} \in E_{N,p}(\bm{S}_{[\widehat{l}]})$
	(Note: such a sequence exists by the denseness of
	$\St(p,N) \setminus E_{N,p}(\bm{S}_{[\widehat{l}]})$
	in
	$\St(p,N)$ [see Fact~\ref{fact:property}~\ref{enum:dense}]).
	In this case,
	$(\bm{V}_{n}^{\heartsuit})_{n=0}^{\infty}$
	is unbounded in
	$Q_{N,p}(\bm{S}_{[\widehat{l}]})$
	by Fact~\ref{fact:property}~\ref{enum:characterize_singular}, and the continuities of
	$f$
	and
	$\Phi_{\bm{S}_{[\widehat{l}]}}^{-1}$
	imply the existence of
	$n' \in \mathbb{N}$
	such that
	$f(\bm{U}^{\star}) \leq f_{\bm{S}_{[\widehat{l}]}}(\bm{V}_{n}^{\heartsuit}) < f_{\bm{S}_{[\widehat{l}]}}(\bm{V}_{\nu})\ (\forall n \geq n')$,
	implying thus the unbounded sequence
	$(\bm{V}_{n}^{\heartsuit})_{n=0}^{\infty}$
	satisfies
	$(\bm{V}_{n}^{\heartsuit})_{n=n'}^{\infty} \subset \mathrm{lev}_{\leq f_{\bm{S}_{[\widehat{l}]}}(\bm{V}_{\nu})}f_{\bm{S}_{[\widehat{l}]}}$
\end{remark}

Theorem~\ref{theorem:convergence} below is the proposed convergence analysis for the ALCP strategy incorporating Type A algorithm.

\begin{theorem}[Convergence analysis for Algorithm~\ref{alg:ALCP} incorporating Type A algorithms] \label{theorem:convergence}
	Let
	$f:\mathbb{R}^{N\times p} \to \mathbb{R}$
	be differentiable and
	$\nabla f_{\bm{S}}$
	Lipschitz continuous, for every
	$\bm{S} \in {\rm O}(N)$,
	with a common Lipschitz constant
	$L>0$ (see~\eqref{eq:Lipschitz} and Fact~\ref{fact:Lipschitz}).
	Let
	$(\bm{V}_{n})_{n=0}^{\infty}$
	be generated by Algorithm~\ref{alg:ALCP} using the alarming condition~\eqref{eq:change} with
	$T > 0$
	in line~\ref{lst:line:if}.
	Assume that
	\begin{enumerate}[label=(\roman*)]
		\item
			Algorithm~\ref{alg:ALCP} incorporates Type A algorithms (see Definition~\ref{definition:line_search_ALCP}).
		\item
			$\underline{\gamma}:=\inf \{\gamma_{\min(\mathcal{N}_{l})} \mid l\in\mathbb{N}_{0}\}>0$
			in~\eqref{eq:update_rule} (see Remark~\ref{remark:stepsize} for the existence of such
			$(\gamma_{n})_{n=0}^{\infty}$).
	\end{enumerate}
	Suppose
	$\|\nabla f_{\bm{S}_{[\ell(n)]}}(\bm{V}_{n})\|_{F} > 0$
	for all
	$n \in \mathbb{N}_{0}$
	(Otherwise, the existence of some
	$m \in \mathbb{N}_{0}$
	satisfying
	$\|\nabla f_{\bm{S}_{[\ell(m)]}}(\bm{V}_{m})\|_{F} = 0$
	is ensured, implying thus the solution of Problem~\ref{problem:ALCP_grad} is achievable in finite updates of Algorithm~\ref{alg:ALCP}),
	where
	$\ell$
	is defined in~\eqref{eq:ln}.
	Then, we have
	\begin{equation} \label{eq:liminf}
		\liminf_{n\to\infty}\|\nabla(f\circ\Phi_{\bm{S}_{[\ell(n)]}}^{-1})(\bm{V}_{n})\|_{F} = 0.
	\end{equation}
\end{theorem}
\begin{proof}
	To simplify our analysis, we divide the behaviors of Algorithm~\ref{alg:ALCP} into the following two cases.
	\begin{enumerate}[label=Case \arabic*:,leftmargin=*,align=left]
		\item\label{enum:case1}
			Center points in
			${\rm O}_{p}(N)$
			are changed infinite times, i.e.,
			$\lim_{n\to\infty} \ell(n) = \infty$.
		\item\label{enum:case2}
			Center points in
			${\rm O}_{p}(N)$
			are changed finite times, i.e.,
			${\displaystyle\widehat{l}:=\max_{n\in \mathbb{N}} \ell(n)} \in \mathbb{N}_{0}$
			exists.
	\end{enumerate}

	\ref{enum:case1}
	From (a) in Definition~\ref{definition:line_search_ALCP}, negative gradients
	$-\nabla f_{\bm{S}_{[l]}}(\bm{V}_{\min(\mathcal{N}_{l})})$
	at
	$\bm{V}_{\min(\mathcal{N}_{l})}\in Q_{N,p}(\bm{S}_{[l]})$
	are employed as search directions for all
	$l \in \mathbb{N}_{0}$.
	By the Armijo condition~\eqref{eq:Armijo} in the update from
	$\bm{V}_{\min(\mathcal{N}_{l})}$
	to
	$\bm{\widetilde{V}}_{\min(\mathcal{N}_{l})+1}$,
	and by
	$\bm{U}_{\min(\mathcal{N}_{l})} = \Phi_{\bm{S}_{[l]}}^{-1}(\bm{V}_{\min(\mathcal{N}_{l})})$
	and
	$\bm{U}_{\min(\mathcal{N}_{l})+1} = \Phi_{\bm{S}_{[l]}}^{-1}(\bm{\widetilde{V}}_{\min(\mathcal{N}_{l})+1})$,
	we have
	\begin{equation}
		(l\in \mathbb{N}_{0}) \quad f(\bm{U}_{\min(\mathcal{N}_{l})+1}) \leq f(\bm{U}_{\min(\mathcal{N}_{l})}) - c\gamma_{\min(\mathcal{N}_{l})}\|\nabla f_{\bm{S}_{[l]}}(\bm{V}_{\min(\mathcal{N}_{l})})\|_{F}^{2}.
	\end{equation}
	Then, by summing up the above from
	$l=0$
	to any
	$q\in\mathbb{N}_{0}$,
	and by letting
	$S_{q}:=\sum_{l=0}^{q}\left(f(\bm{U}_{\min(\mathcal{N}_{l})}) - f(\bm{U}_{\min(\mathcal{N}_{l})+1})\right)$,
	we obtain
	\begin{align} \label{eq:limsup_bounded}
		& \sum_{l=0}^{q} \gamma_{\min(\mathcal{N}_{l})}\|\nabla f_{\bm{S}_{[l]}}(\bm{V}_{\min(\mathcal{N}_{l})})\|_{F}^{2} \leq \frac{1}{c} S_{q}.
	\end{align}
	The LHS is clearly monotone increasing.
	Moreover, it is bounded, hence converged, because
	\begin{align}
		S_{q} & = f(\bm{U}_{0}) - \sum_{l=0}^{q-1}\left(\underbrace{f(\bm{U}_{\min(\mathcal{N}_{l})+1}) - f(\bm{U}_{\min(\mathcal{N}_{l+1})})}_{\geq 0\ (\because \min(\mathcal{N}_{l})+1\leq \min(\mathcal{N}_{l+1})\ \mathrm{from~\eqref{eq:N_l_consecutive}})}\right) -f(\bm{U}_{\min(\mathcal{N}_{q})+1}) \\
					& \leq f(\bm{U}_{0}) - f(\bm{U}_{\min(\mathcal{N}_{q})+1})
					\leq f(\bm{U}_{0}) - \min_{\bm{U}\in\St(p,N)} f(\bm{U}) < \infty,
	\end{align}
	where
	the monotone decrease of
	$(f(\bm{U}_{n}))_{n=0}^{\infty} (= (f_{\bm{S}_{[\ell(n)]}}(\bm{V}_{n}))_{n=0}^{\infty})$,
	which is ensured by the descent condition~\eqref{eq:descent} and the Armijo condition~\eqref{eq:Armijo}, is used. Therefore, we obtain
	\begin{align}
		& \underline{\gamma}\sum_{l=0}^{\infty}\|\nabla f_{\bm{S}_{[l]}}(\bm{V}_{\min(\mathcal{N}_{l})})\|_{F}^{2} \leq
		\sum_{l=0}^{\infty} \gamma_{\min(\mathcal{N}_{l})}\|\nabla f_{\bm{S}_{[l]}}(\bm{V}_{\min(\mathcal{N}_{l})})\|_{F}^{2} \\ 
		& \leq \frac{f(\bm{U}_{0}) - \min_{\bm{U} \in \St(p,N)} f(\bm{U})}{c} < \infty,
	\end{align}
	from which we obtain
	$\lim_{l\to \infty} \|\nabla f_{\bm{S}_{[l]}}(\bm{V}_{\min(\mathcal{N}_{l})})\|_{F}^{2} = 0$.

	\ref{enum:case2}
	Recall
	$\mathcal{N}_{\widehat{l}}= [\min(\mathcal{N}_{\widehat{l}}),\infty) \cap \mathbb{N}_{0}$
	and the boundedness of
	$(\bm{V}_{n})_{n\in \mathcal{N}_{\widehat{l}}}$
	(see Lemma~\ref{lemma:bounded}).
	Recall also that
	$\mathcal{A}^{\langle \widehat{l}\rangle}$
	is a Type A algorithm and
	$\nabla f_{\bm{S}_{[\widehat{l}]}}$
	is Lipschitz continuous over
	$Q_{N,p}(\bm{S}_{[\widehat{l}]})$.
	Then, the condition (d) in Definition~\ref{definition:line_search_ALCP} ensures
	$\liminf_{n\to\infty,\\ n\in \mathcal{N}_{\widehat{l}}}\|\nabla(f\circ\Phi_{\bm{S}_{[\widehat{l}]}}^{-1})(\bm{V}_{n})\|_{F} = 0$.
\end{proof}

\begin{remark}[The existence of $(\gamma_{n})_{n=0}^{\infty}$ satisfying $\underline{\gamma} > 0$ in Theorem~\ref{theorem:convergence}] \label{remark:stepsize}
		Recall that, for every
		$l \in \mathbb{N}_{0}$,
		the Lipschitz continuity of
		$\nabla f_{\bm{S}_{[l]}}$
		on
		$Q_{N,p}(\bm{S}_{[l]})$
		with a Lipschitz constant
		$L > 0$
		ensures
		\begin{equation}
			(\bm{V} \in Q_{N,p}(\bm{S}_{[l]}), \gamma \in \mathbb{R}) \quad
			f_{\bm{S}_{[l]}}(\bm{V}-\gamma\nabla f_{\bm{S}_{[l]}}(\bm{V})) \leq f_{\bm{S}_{[l]}}(\bm{V}) + \left(\frac{\gamma^{2}L}{2} - \gamma\right)\|\nabla f_{\bm{S}_{[l]}}(\bm{V})\|_{F}^{2}.
		\end{equation}
		Then, we can choose
		$\gamma_{\min(\mathcal{N}_{l})} \in \left(0,\frac{2(1-c)}{L}\right]$
		satisfying the Armijo condition~\eqref{eq:Armijo} with some
		$c \in (0,1)$
		via the following inequality
		\begin{align}
			& f_{\bm{S}_{[l]}}(\bm{V}_{\min(\mathcal{N}_{l})}-\gamma_{\min(\mathcal{N}_{l})}\nabla f_{\bm{S}_{[l]}}(\bm{V}_{\min(\mathcal{N}_{l})})) \\
			& \leq
			f_{\bm{S}_{[l]}}(\bm{V}_{\min(\mathcal{N}_{l})}) + \left(\frac{\gamma_{\min(\mathcal{N}_{l})}^{2}L}{2} - \gamma_{\min(\mathcal{N}_{l})}\right)\|\nabla f_{\bm{S}_{[l]}}(\bm{V}_{\min(\mathcal{N}_{l})})\|_{F}^{2} \\
			& \leq f_{\bm{S}_{[l]}}(\bm{V}_{\min(\mathcal{N}_{l})}) + c\gamma_{\min(\mathcal{N}_{l})} \inprod{\nabla f_{\bm{S}_{[l]}}(\bm{V}_{\min(\mathcal{N}_{l})})}{-\nabla f_{\bm{S}_{[l]}}(\bm{V}_{\min(\mathcal{N}_{l})})},
		\end{align}
		where
		$\bm{D}_{\min(\mathcal{N}_{l})} = - \nabla f_{\bm{S}_{[l]}}(\bm{V}_{\min(\mathcal{N}_{l})})$ (see Definition~\ref{definition:line_search_ALCP}~(a)).
		Since the Lipschitz constant
		$L$
		of
		$\nabla f_{\bm{S}}$
		is common for any
		$\bm{S} \in {\rm O}(N)$
		by the assumption in Theorem~\ref{theorem:convergence}, there exists
		$(\gamma_{n})_{n=0}^{\infty}$
		satisfying simultaneously the Armijo condition~\eqref{eq:Armijo} and
		$ \underline{\gamma} > 0$.
		Such
		$(\gamma_{n})_{n=0}^{\infty}$
		can be generated by the standard backtracking algorithm (see, e.g.,~\cite{Nocedal-Wright06,Andrei20}).
\end{remark}

\subsection{Similar framework using retraction} \label{sec:retraction}
As a standard strategy for Problem~\ref{problem:origin}, {\it a retraction-based strategy}~\cite{Absil-Mahony-Sepulchre08} has been developed with the so-called {\it retraction}
$R:T\St(p,N)\to \St(p,N):(\bm{U},\bm{\mathcal{V}}) \mapsto R_{\bm{U}}(\bm{\mathcal{V}})$
satisfying
the restriction
$R_{\bm{U}}$
of
$R$
to
$T_{\bm{U}}\St(p,N)$
for
$\bm{U} \in \St(p,N)$
is differentiable, and
\begin{align}
	 \quad (\bm{U}\in\St(p,N)) & \quad R_{\bm{U}}(\bm{0}) = \bm{U} \\
	 \quad  (\bm{U} \in \St(p,N), \bm{\mathcal{V}} \in T_{\bm{U}}\St(p,N)) & \quad \mathrm{D}R_{\bm{U}}(\bm{0})[\bm{\mathcal{V}}] = \bm{\mathcal{V}}, \label{eq:retraction_2}
\end{align}
where
$T\St(p,N):= \bigcup_{\bm{U}\in\St(p,N)} \{\bm{U}\}\times T_{\bm{U}}\St(p,N)$
is the tangent bundle of
$\St(p,N)$
and
$T_{\bm{U}}\St(p,N):=\{\bm{\mathcal{D}} \in \mathbb{R}^{N\times p} \mid \bm{U}^{\TT}\bm{\mathcal{D}}+\bm{\mathcal{D}}^{\TT}\bm{U} = \bm{0}\}$
is the tangent space at
$\bm{U}\in \St(p,N)$
to
$\St(p,N)$.
Several retractions have been introduced for
$\St(p,N)$,
e.g., the Riemannian exponential mapping for
$\St(p,N)$,
the QR decomposition-based retraction, the polar decomposition-based retraction~\cite{Absil-Mahony-Sepulchre08}, and the Cayley transform-based retraction~\cite{Wen-Yin13}.

For every retraction, the G\^{a}teaux derivative of
$f\circ R_{\bm{U}}$
at
$\bm{0}  \in T_{\bm{U}}\St(p,N)$
along
$\bm{\mathcal{D}} \in T_{\bm{U}}\St(p,N)$
with
$\bm{U} \in \St(p,N)$
implies
\begin{align}
	& (t \in \mathbb{R}) \quad f(R_{\bm{U}}(t\bm{\mathcal{D}})) = f\circ R_{\bm{U}}(\bm{0}) + t \mathrm{D}(f\circ R_{\bm{U}})(\bm{0})[\bm{\mathcal{D}}] + o(t) \\
	& = f(\bm{U}) + t \mathrm{D}f(R_{\bm{U}}(\bm{0}))[\mathrm{D}R_{\bm{U}}(\bm{0})[\bm{\mathcal{D}}]] + o(t)
	= f(\bm{U}) + t \mathrm{D}f(\bm{U})[\bm{\mathcal{D}}] + o(t)\ (\because~\eqref{eq:retraction_2}) \\
	& = f(\bm{U}) + t \inprod{\mathop{\mathrm{grad}}f(\bm{U})}{\bm{\mathcal{D}}}_{\bm{U}}  + o(t), \label{eq:approximation} 
\end{align}
where
$o(\cdot)$
denotes Landau's little-o notation, i.e.,
$\lim_{t \to 0} |o(t)/t|= 0$,
$\inprod{\cdot}{\cdot}_{\bm{U}}$
denotes an inner product on
$T_{\bm{U}}\St(p,N)$,
and
$\mathop{\mathrm{grad}}f(\bm{U})\in T_{\bm{U}}\St(p,N)$
denotes the Riemannian gradient\footnote{
	The Riemannian gradient is defined uniquely under a given inner product
	$\inprod{\cdot}{\cdot}_{\bm{U}}$
	on
	$T_{\bm{U}}\St(p,N)$
	as a vector
	$\mathop{\mathrm{grad}}f(\bm{U}) \in T_{\bm{U}}\St(p,N)$
	satisfying
	$\mathrm{D}f(\bm{U})[\bm{\mathcal{D}}] = \inprod{\mathrm{grad}f(\bm{U})}{\bm{\mathcal{D}}}_{\bm{U}}$
	for all
	$\bm{\mathcal{D}} \in T_{\bm{U}}\St(p,N)$ (see, e.g.,\cite{Boumal20}).
	In particular, under the so-called canonical inner product~\cite{Edelman-Arias-Smith98} for
	$\St(p,N)$ defined as
	$\inprod{\bm{\mathcal{D}}_{1}}{\bm{\mathcal{D}}_{2}}_{\bm{U}} := \trace\left(\bm{\mathcal{D}}_{1}^{\TT}\left(\bm{I}-\frac{1}{2}\bm{U}\bm{U}^{\TT}\right)\bm{\mathcal{D}}_{2}\right)\ (\bm{\mathcal{D}}_{1},\bm{\mathcal{D}}_{2}\in T_{\bm{U}}\St(p,N))$,
	we can compute
	$\mathop{\mathrm{grad}}f(\bm{U}) (= \nabla (f\circ R_{\bm{U}})(\bm{0})) = \nabla f(\bm{U}) - \bm{U}\nabla f(\bm{U})^{\TT}\bm{U}$~\cite{Wen-Yin13},
	where
	$\nabla f(\bm{U}) \in \mathbb{R}^{N\times p}$
	is the gradient of
	$f$
	at
	$\bm{U}$
	under the Euclidean inner product.
}.
Via~\eqref{eq:approximation}, in the retraction-based strategy, we find a descent direction
$\bm{\mathcal{D}}_{n} \in T_{\bm{U}_{n}}\St(p,N)$
satisfying
$\inprod{\mathop{\mathrm{grad}}f(\bm{U}_{n})}{\bm{\mathcal{D}}_{n}}_{\bm{U}_{n}} < 0$
like Type A algorithms (see Definition~\ref{definition:line_search}~(b)), and update an estimate
$\bm{U}_{n+1} := R_{\bm{U}_{n}}(\bm{0}+\gamma_{n}\bm{\mathcal{D}}_{n}) \in \St(p,N)$,
with a stepsize
$\gamma_{n} >0$,
of a stationary point
$\bm{U}^{\star} \in \St(p,N)$,
i.e.,
$\mathrm{grad} f(\bm{U}^{\star}) = \bm{0}$~\cite{Wen-Yin13,Gao-Liu-Chen-Yuan18},
of Problem~\ref{problem:origin} (see Remark~\ref{remark:stationary}~(a)).
Along this strategy, several Euclidean optimization algorithms have been extended for Problem~\ref{problem:origin} (e.g., the gradient descent method~\cite{Edelman-Arias-Smith98,Nikpour02,Nishimori-Akaho05,Abrudan-Eriksson-Koivunen,Absil-Mahony-Sepulchre08}, Newton's method~\cite{Edelman-Arias-Smith98,Manton15}, and the trust-region method~\cite{Absil-Baker-Gallivan07,Kasai-Mishra18}).

Euclidean optimization algorithms exploiting information about the past iterates (e.g., the past search directions) are known to enjoy fast numerical convergence properties.
Such algorithms include, e.g., the quasi-Newton method~\cite{Nocedal-Wright06}, the conjugate gradient method~\cite{Andrei20,Gilbert-Nocedal92,Al-baali85,Dai-Yuan99,Dai-etal00,Dai-Yuan01,Hager-Zhang05}, the three-term conjugate gradient method~\cite{Zhang-Zhou-Li07,Narushima-Yabe-Ford11,Khoshsimaye-Ashrafi23},  and the accelerated gradient method~\cite{Nesterov83,Ghadimi-Lan16,Carmon-Duchi-Hinder-Sidford18,Zeyuan18,Diakonikolas-Jordan21}.
Extensions of such algorithms along the retraction-based strategy are not simple because the past search directions
$(\bm{\mathcal{D}}_{k})_{k=0}^{n-1}$
are designed on the past tangent spaces, implying thus
$(\bm{\mathcal{D}}_{k})_{k=0}^{n-1}$
can not be utilized directly on the current tangent space
$T_{\bm{U}_{n}}\St(p,N)$.
To resolve this issue, the so-called {\it vector transport}~\cite{Absil-Mahony-Sepulchre08} have been used for a translation of
$(\bm{\mathcal{D}}_{k})_{k=0}^{n-1}$
into
$T_{\bm{U}_{n}}\St(p,N)$,
and some conjugate gradient methods~\cite{Ring-Wirth12,Sato-Iwai15,Zhu17,Zhu-Sato20,Sakai-Iiduka20,Sato21B,Sakai-Iiduka-Sato23} and the quasi-Newton method~\cite{Ring-Wirth12, Huag-Gallivan-Absil15} have been extended with a vector transport.
However, there still remain many Euclidean optimization algorithms (e.g., HS+, PRP+, and LS+-type conjugate gradient methods~\cite{Andrei20}, three-term conjugate gradient methods~\cite{Zhang-Zhou-Li07,Narushima-Yabe-Ford11,Khoshsimaye-Ashrafi23}, and accelerated gradient methods~\cite{Ghadimi-Lan16,Carmon-Duchi-Hinder-Sidford18,Zeyuan18,Diakonikolas-Jordan21}), with a lot of potential, for which any practical way does not seem to have been reported for extensions along the retraction-based strategy with vector transports.

Recently, as an instance of adaptive parametrization strategies with a retraction, {\it a dynamic trivialization} for Problem~\ref{problem:origin} has been introduced, e.g.,~\cite{Lezcano19,Criscitiello-Boumal19,Lezcano20,Criscitiello-Boumal22}, for direct utilization of Euclidean optimization algorithms.
A motivation of the dynamic trivialization seems to be common as that of the ALCP strategy in Section~\ref{sec:adaptive}, which is available to enjoy Euclidean optimization algorithms as far as not facing any computational difficulty or any performance degradation.
In a way similar to the ALCP strategy, Problem~\ref{problem:origin} can be tackled by reformulating as
\begin{equation} \label{eq:trivialization}
	{\rm find}\ (\bm{\mathcal{V}}^{\star},\bm{U}^{\star}) \in T_{\bm{U}^{\star}}\St(p,N)\times \St(p,N)\ {\rm s.t.}\ \|\nabla (f\circ R_{\bm{U}^{\star}}) (\bm{\mathcal{V}}^{\star})\| = 0,
\end{equation}
corresponding to Problem~\ref{problem:ALCP_grad}.
By fixing a tangent point
$\bm{U} \in \St(p,N)$
of
$T_{\bm{U}}\St(p,N)$
for
$R_{\bm{U}}$,
the dynamic trivialization~\cite{Lezcano19,Criscitiello-Boumal19,Lezcano20,Criscitiello-Boumal22} updates an estimate
$\bm{\mathcal{V}}_{n} \in T_{\bm{U}}\St(p,N)$
of a stationary point
$\bm{\mathcal{V}}^{\star}$
of
$f\circ R_{\bm{U}}$
with a Euclidean optimization algorithm.
However, as
$\bm{\mathcal{V}}_{n}$
deviates distant from
$\bm{0} \in T_{\bm{U}}\St(p,N)$,
difficulties for finding a stationary point of
$f\circ R_{\bm{U}}$
likely appear (see Remark~\ref{remark:difficulty_trivialization} for the difficulties).
Just after facing such a difficulty at
$\bm{\mathcal{V}}_{n} \in T_{\bm{U}}\St(p,N)$,
the tangent point
$\bm{U}$
is changed to
$\bm{U}':= R_{\bm{U}}(\bm{\mathcal{V}}_{n}) \in \St(p,N)$
in order to mitigate the difficulty, and repeat the update of
$\bm{\mathcal{V}}_{n+k} \in T_{\bm{U}'}\St(p,N)\ (k\geq1)$
for finding a stationary point of
$f\circ R_{\bm{U}'}$
over
$T_{\bm{U}'}\St(p,N)$
until facing further difficulties.
In the following, we discuss distinctions between the dynamic trivialization and the ALCP strategy.

\begin{remark}[Difficulties facing at $\bm{\mathcal{V}}_{n}\in T_{\bm{U}}\St(p,N)$ deviates distant from $\bm{0}$] \label{remark:difficulty_trivialization}
	\mbox{}
\begin{enumerate}[label=(\alph*)]
	\item
		There is a risk caused by the fact that the diffeomorphism of
		$R_{\bm{U}}$
		is guaranteed only within some open ball centered at
		$\bm{0} \in T_{\bm{U}}\St(p,N)$~\cite{Absil-Mahony-Sepulchre08}.
		The violation of diffeomorphism of
		$R_{\bm{U}}$
		can introduce extra stationary points, meaning that
		$R_{\bm{U}}(\bm{\mathcal{V}}_{n}) \in \St(p,N)$
		is not a stationary point of Problem~\ref{problem:origin} even if
		$(\bm{\mathcal{V}}_{n},\bm{U})$
		is a solution of the problem~\eqref{eq:trivialization}~\cite{Lezcano19,Lezcano20}.
		In contrast, for every
		$\bm{S} \in {\rm O}(N)$,
		$\Phi_{\bm{S}}^{-1}$
		in the proposed parametrization is diffeomorphic over
		$Q_{N,p}(\bm{S})$
		entirely.
	\item
		There exists a risk of loosing desired satisfactions of required conditions for Euclidean optimization algorithms applied to
		$f\circ R_{\bm{U}}$
		over
		$T_{\bm{U}}\St(p,N)$.
		Such conditions include, e.g., the Lipschitz continuity of
		$\nabla (f\circ R_{\bm{U}})$~\cite{Criscitiello-Boumal19,Lezcano20,Criscitiello-Boumal22} is guaranteed only on some open ball centered at
		$\bm{0} \in T_{\bm{U}}\St(p,N)$
		in general even if
		$\nabla f$
		is Lipschitz continuous on
		$\St(p,N)$.
		In contrast, for every
		$\bm{S} \in {\rm O}(N)$,
		$\nabla (f\circ\Phi_{\bm{S}}^{-1})$
		in the proposed parametrization is Lipschitz continuous over
		$Q_{N,p}(\bm{S})$
		entirely if
		$\nabla f$
		is Lipschitz continuous on
		$\St(p,N)$
		(see Fact~\ref{fact:Lipschitz}).
	\item
		If the retraction
		$R_{\bm{U}}$
		is not a surjection onto
		$\St(p,N)$,
		then the existence of a stationary point of
		$f\circ R_{\bm{U}}$
		in
		$T_{\bm{U}}\St(p,N)$
		may not be guaranteed.
		This difficulty corresponds to the existence of the singular-point set
		$E_{N,p}(\bm{S})$
		for
		$\Phi_{\bm{S}}$ (see~\eqref{eq:singular}).
\end{enumerate}

\end{remark}

\subsubsection{Dynamic trivialization of ${\rm O}(N)$ with the Cayley transform-based retraction} \label{sec:Cayley_retraction}
	For the problem~\eqref{eq:trivialization} with
	$p=N$,
	i.e.,
	in case of
	$\St(N,N) = {\rm O}(N)$,
	a special instance of the dynamic
	trivialization has been proposed in~\cite{Lezcano19} with the Cayley transform-based retraction~\cite{Wen-Yin13}
	\begin{equation}
		(\bm{\mathcal{V}} \in T_{\bm{U}}\St(p,N)) \ R_{\bm{U}}^{\rm Cay}(\bm{\mathcal{V}}) := \varphi^{-1}\left(\Skew\left(\bm{U}\bm{\mathcal{V}}^{\TT}\left(\bm{I}-\frac{1}{2}\bm{U}\bm{U}^{\TT}\right)\right)\right)\bm{U},
	\end{equation}
	where
	$\varphi^{-1}$
	is defined in~\eqref{eq:ICT_O}.
	In the dynamic trivialization with
	$R_{\bm{U}}^{\rm Cay}$
	for optimization over
	${\rm O}(N)$
	in~\cite{Lezcano19},
	the tangent point
	$\bm{U}$
	for
	$R_{\bm{U}}^{\rm Cay}$
	is changed after constant number, say
	$\kappa \in \mathbb{N}$,
	of updates of
	$\bm{\mathcal{V}}_{n}$
	mainly because of difficulty in Remark~\ref{remark:difficulty_trivialization}~(c).
	This dynamic trivialization can be seen as a special instance of the ALCP strategy in Algorithm~\ref{alg:ALCP} with a simple alarming condition, in line~\ref{lst:line:if}, which is satisfied at
	$n=\kappa + \nu (= \kappa + \min(\mathcal{N}_{l}))$
	because
	$R_{\bm{U}}^{\rm Cay}$
	can be expressed in terms of
	$\Phi_{[\bm{U}\ \bm{U}_{\perp}]}^{-1}$~\cite{Kume-Yamada22}
	as
	\begin{equation} \label{eq:Cayley_retraction}
		(\bm{\mathcal{V}} \in T_{\bm{U}}\St(p,N)) \ R_{\bm{U}}^{\rm Cay}(\bm{\mathcal{V}})
		=\Phi_{[\bm{U}\ \bm{U}_{\perp}]}^{-1} \circ \Psi_{[\bm{U}\ \bm{U}_{\perp}]}(\bm{\mathcal{V}}),
	\end{equation}
	where
	$\bm{U}_{\perp} \in \St(N-p,N)$
	satisfies
	$\bm{U}^{\TT}\bm{U}_{\perp} = \bm{0}$
	and
	\begin{equation}
		\Psi_{[\bm{U}\ \bm{U}_{\perp}]}:T_{\bm{U}}\St(p,N) \to Q_{N,p}([\bm{U}\ \bm{U}_{\perp}]):\bm{\mathcal{V}}\mapsto
					-\frac{1}{2}\begin{bmatrix} \bm{U}^{\TT}\bm{\mathcal{V}} & - (\bm{U}_{\perp}^{\TT}\bm{\mathcal{V}})^{\TT} \\ \bm{U}_{\perp}^{\TT}\bm{\mathcal{V}} & \bm{0}\end{bmatrix} \label{eq:tangent_to_Q}
	\end{equation}
	is an invertible linear operator with its inversion
	$\Psi_{[\bm{U}\ \bm{U}_{\perp}]}^{-1}:Q_{N,p}([\bm{U}\ \bm{U}_{\perp}]) \to T_{\bm{U}}\St(p,N):\bm{V}\mapsto -2[\bm{U}\ \bm{U}_{\perp}]\bm{V}\bm{I}_{N\times p}$.

	By comparison to the dynamic trivialization with
	$R_{\bm{U}}^{\rm Cay}$
	for the problem~\eqref{eq:trivialization} with
	$p = N$
	in~\cite{Lezcano19}, the ALCP strategy has mainly three advantages:
	(i) although any convergence analysis for the dynamic trivialization with
	$R_{\bm{U}}^{\rm Cay}$
	does not seem to have been reported, that for the ALCP strategy is given in Section~\ref{sec:convergence};
	(ii) the ALCP strategy can be applied to Problem~\ref{problem:ALCP_grad} for general
	$p \leq N$;
	(iii) the changing condition for
	$\bm{S} \in {\rm O}(N)$
	can be designed more flexibly, e.g., by exploiting the condition~\eqref{eq:change} in Algorithm~\ref{alg:ALCP}, than the changing condition for the tangent point
	$\bm{U}$
	in
	$R_{\bm{U}}$,
	i.e., the change is performed after prefixed number of updates of
	$\bm{\mathcal{V}}_{n}$
	for estimating
	$\bm{\mathcal{V}}^{\star}$.
	Indeed, the latter naive changing condition in~\cite{Lezcano19} does not seem to enjoy maximally the potential of each tangent space
	$T_{\bm{U}}\St(p,N)$
	where the Euclidean optimization algorithm is executed for minimization of
	$f\circ R_{\bm{U}}$.

\subsubsection{Dynamic trivialization with the Riemannian Exponential mapping} \label{sec:exponential}
For finding a stationary point of optimization over a general Riemannian manifold, the dynamic trivialization with the Riemannian exponential mapping has been proposed in~\cite{Criscitiello-Boumal19,Lezcano20,Criscitiello-Boumal22}.
Since any Riemannian manifold has the unique Riemannian exponential mapping, the dynamic trivialization with Riemannian exponential mapping can be applied to optimization over any Riemannian manifold.
The Riemannian exponential mapping with the so-called canonical metric for
$\St(p,N)$
is given~\cite{Edelman-Arias-Smith98} as
\begin{equation}
	(\bm{U}\in \St(p,N), \bm{\mathcal{D}} \in T_{\bm{U}}\St(p,N)) \ 
	\mathrm{Exp}_{\bm{U}}(\bm{\mathcal{D}}) :=
	\begin{bmatrix} \bm{U} & \bm{Q}  \end{bmatrix}
	\exp_{\mathrm{m}}
	\begin{bmatrix} \bm{U}^{\TT}\bm{\mathcal{D}} & -\bm{R}^{\TT} \\ \bm{R} & \bm{0} \end{bmatrix} 
	\bm{I}_{N\times p},
\end{equation}
where
$\bm{Q}\bm{R} = (\bm{I}-\bm{U}\bm{U}^{\TT})\bm{\mathcal{D}}$
is the QR decomposition and
$\exp_{\mathrm{m}}(\bm{X}):=\sum_{k=0}^{\infty} \bm{X}^{k}/k!$,
for
$\bm{X} \in \mathbb{R}^{2p\times 2p}$,
is the matrix exponential mapping.

In the dynamic trivialization with
$\mathrm{Exp}_{\bm{U}}$
for the problem~\eqref{eq:trivialization}, the tangent point
$\bm{U} \in \St(p,N)$
is changed if the current estimate
$\bm{\mathcal{\bm{V}}}_{n} \in T_{\bm{U}}\St(p,N)$
deviates from a closed ball in
$T_{\bm{U}}\St(p,N)$
centered at
$\bm{0} \in T_{\bm{U}}\St(p,N)$
with a radius
$r > 0$
designed according to
$\St(p,N)$
because of difficulties in Remark~\ref{remark:difficulty_trivialization}~(a) and (b)~\cite{Criscitiello-Boumal19,Lezcano20,Criscitiello-Boumal22}.

Comparisons with the dynamic trivialization with
$\mathrm{Exp}_{\bm{U}}$
and the ALCP strategy are summarized as follows:
\begin{enumerate}[label=(\roman*)]
	\item
		Convergence analyses of the existing dynamic trivializations with
		$\mathrm{Exp}_{\bm{U}}$
		have been reported for very limited cases of Euclidean optimization algorithms (e.g., the gradient descent method~\cite{Lezcano20}, and an accelerated gradient method~\cite{Criscitiello-Boumal22}).
		In contrast, a convergence analysis, in Section~\ref{sec:convergence}, of the proposed ALCP strategy can be applied, in a unified way, to a wider class of Euclidean optimization algorithms.
	\item
		In the dynamic trivialization with
		$\mathrm{Exp}_{\bm{U}}$,
		the computation of the gradient
		$\nabla (f\circ\mathrm{Exp}_{\bm{U}})$
		requires the derivative of
		$\mathrm{Exp}_{\bm{U}}$,
		which has been computed by iterative algorithms~\cite{Al-Higham09} because its closed-form expression has not been known.
		In contrast, the gradient
		$\nabla (f\circ\Phi_{\bm{S}}^{-1})$
		can be computed within finite matrix calculations~\cite{Kume-Yamada22} (see Fact~\ref{fact:gradient}).
\end{enumerate}

\section{Numerical experiments}\label{sec:numerical}
We demonstrate the performance of the proposed ALCP strategy in Algorithm~\ref{alg:ALCP} by numerical experiments.
We implemented Algorithm~\ref{alg:ALCP} incorporating Type A algorithms with the alarming condition~\eqref{eq:change} ($T=1.5$) in MATLAB.
For comparisons, we used retraction-based optimization algorithms implemented in Manopt~\cite{Boumal-Mishra-Absil-Sepulchre14}, a MATLAB toolbox for Riemannian optimization.
As a retraction, we used the QR decomposition-based retraction~\cite{Absil-Mahony-Sepulchre08} because this is known as the most computationally efficient retraction~\cite{Zhu17}.
All the experiments were performed in MATLAB on MacBook Pro (13-inch, M1, 2020) with 16GB of RAM.

We used the standard backtracking algorithm, e.g.,~\cite{Nocedal-Wright06}, for both algorithms to estimate a stepsize satisfying the Armijo condition~\eqref{eq:Armijo_Euclidean}.
Algorithm~\ref{alg:backtracking} illustrates the backtracking algorithm for a given differentiable
$J:\mathcal{X}\to \mathbb{R}$
defined over a Euclidean space
$\mathcal{X}$.
In Algorithm~\ref{alg:ALCP} incorporating Type A algorithms, we used Algorithm~\ref{alg:backtracking} with
$J:=f\circ\Phi_{\bm{S}_{[l]}}^{-1}$,
$\mathcal{X}:=Q_{N,p}(\bm{S}_{[l]})$,
$\bm{x}:=\bm{V}_{n}$,
and
$\bm{d}:=\bm{D}_{n} \in Q_{N,p}(\bm{S})$
at $n$th iteration.
In the retraction-based strategy, we used Algorithm~\ref{alg:backtracking} with
$J:=f\circ R_{\bm{U}_{n}}$,
$\mathcal{X}:= T_{\bm{U}_{n}}\St(p,N)$,
$\bm{x}:= \bm{0}$,
and
$\bm{d}:=\bm{\mathcal{D}}_{n} \in T_{\bm{U}_{n}}\St(p,N)$
at $n$th iteration, where
$\bm{U}_{n} \in \St(p,N)$
is the $n$th estimate of a solution to Problem~\ref{problem:origin}.
For Algorithm~\ref{alg:backtracking}, we used the default parameters, employed in Manopt, as
$c = 2^{-13},\rho = 0.5$,
and an initial stepsize at $n$th iteration as
\begin{equation}
	\gamma_{\rm initial} :=
	\begin{cases}
		\frac{1}{\|\nabla J(\bm{x}_{0})\|_{F}} & (n = 0)\\
		\frac{4(J(\bm{x}_{n})-J(\bm{x}_{n-1}))}{\inprod{\nabla J(\bm{x}_{n})}{\bm{d}_{n}}} & (n > 0).
	\end{cases}
\end{equation}

The iterations of Algorithm~\ref{alg:ALCP} and the retraction-based algorithms were terminated when
$\frac{\|\bm{G}_{n}\|_{F}}{\|\bm{G}_{0}\|_{F}} < 10^{-5} \quad {\rm or} \quad n=2000$
were achieved, where
\begin{equation} \label{eq:gradient_norm}
	\bm{G}_{n} := \begin{cases} \nabla (f\circ\Phi_{\bm{S}_{[l]}}^{-1})(\bm{V}_{n}) & \textrm{(Algorithm~\ref{alg:ALCP})}\\
	\mathop{\mathrm{grad}}f(\bm{U}_{n}) & \textrm{(retraction-based algorithm)} \end{cases}
\end{equation}

\begin{algorithm}[t]
	\caption{Backtracking algorithm}
	\label{alg:backtracking}
	\begin{algorithmic}
		\Require
		$c\in (0,1),\ \rho \in (0, 1),\ \gamma_{\rm initial} > 0, \ \bm{x}, \bm{d} \in \mathcal{X}$
		\State
			$\gamma \leftarrow \gamma_{\rm initial}$
			\While{$J(\bm{x} + \gamma \bm{d}) > J(\bm{x}) + c\gamma \inprod{\nabla J(\bm{x})}{\bm{d}} $}
				\State
				$\gamma \leftarrow \rho \gamma$
			\EndWhile
		\Ensure
		$\gamma$
	\end{algorithmic}
\end{algorithm}

\subsection{Avoidance of the singular-point issue}
In this subsection, we tested how dramatically the changing scheme of center points in the ALCP strategy can improve the convergence speed of the naive CP strategy by considering the following toy problem:
\begin{equation}
	\mathop{\mathrm{minimize}}_{\bm{U}\in \St(p,N)} f_1(\bm{U}):=\frac{1}{2}\|\bm{U}-\bm{U}^{\star}\|_{F}^{2}, \label{eq:toyproblem}
\end{equation}
where
$\bm{U}^{\star} \in \St(p,N)$
is its global minimizer.

We compared the gradient descent methods (GDM) along with the naive CP strategy~\cite{Kume-Yamada22} (CP), the proposed ALCP strategy (ALCP) in Algorithm~\ref{alg:ALCP}, and the retraction-based strategy (QR)~\cite{Absil-Mahony-Sepulchre08}.
Recall that the GDM is a simple Type A algorithm for minimization of a differentiable
$J:\mathcal{X}\to\mathbb{R}$
over the Euclidean space
$\mathcal{X}$ (see, e.g.,~\cite{Nocedal-Wright06}),
where a search direction at $n$th update is employed as
$\bm{d}_{n} = -\nabla J(\bm{x}_{n})$.
We note that the naive CP strategy is the specialization of Algorithm~\ref{alg:ALCP} by letting the condition in line~\ref{lst:line:if} be always false.

To demonstrate the improvement, we used a setting causing the singular-point issue in the naive CP strategy as
$\bm{S}:=\bm{I}$
and
$\bm{U}^{\star} := \diag(\bm{R}(127\pi/128),\bm{I}_{N-2})\bm{I}_{N\times p}$,
where
$\bm{R}(\theta):= \begin{bmatrix} \cos(\theta) & -\sin(\theta) \\ \sin(\theta) & \cos(\theta) \end{bmatrix} \in {\rm O}(2)$
is a rotation matrix.
Indeed, by
$\bm{S}_{\rm le}^{\TT}\bm{U}^{\star} = \diag(\bm{R}(127\pi/128),\bm{I}_{p-2})$,
we have
$\det(\bm{I}_{p} + \bm{S}_{\rm le}^{\TT}\bm{U}^{\star}) = 2^{p-1}(1+\cos(127\pi/128)) \approx 2^{p-1}\times 3.0\times 10^{-4}$,
implying thus
$\bm{U}^{\star}$
is close to a singular-point of
$\Phi_{\bm{S}}$.
For every trial, we generated an initial estimate randomly by MATLAB code 'orth(rand(N,p))'.

Table~\ref{table:toyproblem} illustrates average results for
$100$
trials of each algorithm with
$N=1000$
and
$p=10$
for the problem~\eqref{eq:toyproblem}.
In the table, for each output
$\bm{U}_{m} \in \St(p,N)$,
'fval' means the value
$f(\bm{U}_{m})-f(\bm{U}^{\star})$,
'feasi' means the feasibility
$\|\bm{I}_{p}-\bm{U}_{m}^{\TT}\bm{U}_{m}\|_{F}$,
'nrmg' means the norm
$\|\bm{G}_{m}\|_{F}$ (see~\eqref{eq:gradient_norm}),
'itr' means the number of iterations, 'time' means the CPU time (s), and 'change' means the number of changing center points for the ALCP strategy.
Figure~\ref{fig:toyproblem} shows the convergence history of algorithms.
The plots show CPU time on the horizontal axis versus the value
$f(\bm{U})$
on the vertical axis.

\afterpage{
\begin{table}[t]
	\caption{Performance of each algorithm applied to the problem~\eqref{eq:toyproblem}.}
	\centering
	\footnotesize
	\begin{filecontents*}{data/toyproblem_1000_10_20234711120250056.csv}
Algorithm,fval,feasi,nrmg,itr,nfe,time,change
GDM+CP,4.11E-02,1.48E-15,6.78E-03,2000.00,3509.85,2.43,-
GDM+ALCP,6.41E-12,4.83E-15,1.11E-05,17.96,58.48,0.04,1.00
GDM+QR,1.74E-10,1.28E-15,1.66E-05,12.57,44.96,0.05,-
\end{filecontents*}
\csvautobooktabular{data/toyproblem_1000_10_20234711120250056.csv}
	\label{table:toyproblem}
\end{table}
\begin{figure}[h]
	\centering
	\includegraphics[clip, width=0.6\textwidth]{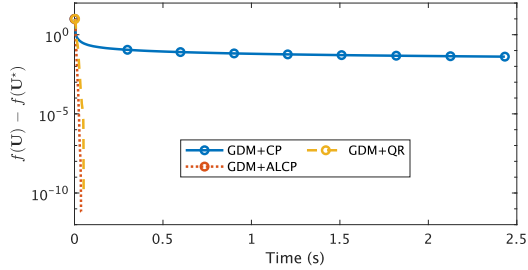}
	\caption{Convergence history of each algorithm applied to the problem~\eqref{eq:toyproblem}.
		Markers are put at every 250 iterations.}
	\label{fig:toyproblem}
\end{figure}
}

Figure~\ref{fig:toyproblem} shows that the ALCP strategy dramatically outperformed the naive CP strategy.
From Table~\ref{table:toyproblem}, the ALCP strategy changed the center point for one time.
These observations imply that the adaptive changing center points in lines~\ref{lst:line:change_start}-\ref{lst:line:change_end} of Algorithm~\ref{alg:ALCP} can improve the slow convergence of the naive CP strategy caused by the singular-point issue.
Moreover, Figure~\ref{fig:toyproblem} demonstrates that the convergence speed of the ALCP strategy was competitive with that of the retraction-based strategy.

\subsection{Comparisons to the retraction-based strategy} \label{sec:numerical_retraction}
We compared numerical performance of the ALCP strategy and the retraction-based strategy by employing three conjugate gradient methods (CGM), known to achieve numerically faster convergence than the GDM.
For minimization of a differentiable
$J:\mathcal{X}\to\mathbb{R}$
over the Euclidean space
$\mathcal{X}$
and an interval subset
$\mathcal{N} \subset \mathbb{N}_{0}$,
the CGM generates a sequence
$(\bm{x}_{n})_{n\in\mathcal{N}}$
by
\begin{equation} \label{eq:CGM}
	(n\in \mathcal{N}) \quad \bm{x}_{n+1} := \bm{x}_{n} + \gamma_{n}\bm{d}_{n}, \quad
	\bm{d}_{n+1} := -\nabla J(\bm{x}_{n+1}) + \beta_{n}\bm{d}_{n},
	\quad \bm{d}_{\min(\mathcal{N})} := -\nabla J(\bm{x}_{\min(\mathcal{N})})
\end{equation}
with a stepsize
$\gamma_{n} > 0$,
a search direction
$\bm{d}_{n} \in \mathcal{X}$,
and
$\beta_{n} \in \mathbb{R}$~\cite{Andrei20}.
Several parameters
$\beta_{n}$
have been proposed to improve convergence behavior~\cite{Andrei20,Gilbert-Nocedal92,Al-baali85,Dai-Yuan99,Dai-etal00,Dai-Yuan01,Hager-Zhang05}.
In our experiments, we employed typical parameters as
\begin{align}
	\beta_{n}^{\rm FR}:= \frac{\inprod{\bm{g}_{n+1}}{\bm{g}_{n+1}}}{\inprod{\bm{g}_{n}}{\bm{g}_{n}}},\ 
	\beta_{n}^{\rm HS+}:= \max\left\{\beta_{n}^{\rm HS},0\right\},\ 
	\beta_{n}^{\rm HZ}:= \max\{\widehat{\beta_{n}^{\rm HZ}},\zeta_{n}\},
\end{align}
where
$\bm{g}_{n}:=\nabla J(\bm{x}_{n})$,
$\bm{y}_{n}:=\bm{g}_{n+1}-\bm{g}_{n}$,
$\beta_{n}^{\rm HS}:=\frac{\inprod{\bm{g}_{n+1}}{\bm{y}_{n}}}{\inprod{\bm{d}_{n}}{\bm{y}_{n}}}$,
$\widehat{\beta_{n}^{\rm HZ}}:= \beta_{n}^{\rm HS} - 2\frac{\|\bm{y}_{n}\|^{2}\inprod{\bm{d}_{n}}{\bm{g}_{n+1}}}{(\inprod{\bm{y}_{n}}{\bm{d}_{n}})^{2}}$
and
$\zeta_{n}:= -\frac{1}{\|\bm{d}_{n}\|\min(0.01,\|\bm{g}_{n}\|)}$.
By letting strategic information at each update from
$\bm{x}_{n} \in \mathcal{X}$
to
$\bm{x}_{n+1} \in \mathcal{X}$
be
\begin{equation}
	\mathfrak{R}_{[\min(\mathcal{N}),n]}:= \begin{cases}
		-\nabla J(\bm{x}_{\min(\mathcal{N})}) & (n=\min(\mathcal{N})) \\
		\bm{d}_{n} & (n>\min(\mathcal{N})),
	\end{cases} \label{eq:record}
\end{equation}
the CGM can be seen as a special example of
$\mathcal{A}$
in~\eqref{eq:Euclidean_optimization}
in Remark~\ref{remark:Euclidean_optimization_algorithm}.

The global convergence
$\liminf_{n\to\infty}\|\nabla J(\bm{x}_{n})\|=0$
for the CGM, with
$\beta_{n}^{\rm FR}$~\cite{Al-baali85},
$\beta_{n}^{\rm HS+}$~\cite{Gilbert-Nocedal92,Dai-Yuan99},
and
$\beta_{n}^{\rm HZ}$~\cite{Hager-Zhang05},
is guaranteed
when every search direction
$\bm{d}_{n}$
satisfies the descent condition~\eqref{eq:descent} and every stepsize
$\gamma_{n}$
is chosen to satisfy the (strong) Wolfe condition, which is stronger than the Armijo condition~\eqref{eq:Armijo}.
Although there is no guarantee that the HS+-type CGM will satisfy the descent condition, a certain restart scheme makes the HS+-type CGM guarantee the descent condition.
Thus, these CGM are Type A algorithms if necessary employing such a restart scheme.

We employed FR, HS+, HZ-type CGM in Algorithm~\ref{alg:ALCP}.
Since these CGM are Type A algorithms, Theorem~\ref{theorem:convergence} guarantees the global convergence when each stepsize satisfies the (strong) Wolfe condition.
For the retraction-based strategy, these CGM have been extended~\cite{Sato21B,Sakai-Iiduka-Sato23} by exploiting {\it a vector transport}~\cite{Absil-Mahony-Sepulchre08} (see Section~\ref{sec:retraction}).
The global convergence
$\liminf_{n\to\infty} \|\mathop{\mathrm{grad}}f(\bm{U}_{n})\|_{F} = 0$
is guaranteed for the CGM with
$\beta_{n}^{\rm FR}$
when each stepsize satisfies a strong Wolfe-type condition~\cite{Sato21B}.
Although, for a general differentiable
$f$,
any global convergence for the CGM with
$\beta_{n}^{\rm HS+}$
and
$\beta_{n}^{\rm HZ}$
has not been reported even if each stepsize satisfies such a Wolfe-type condition\footnote{For a strongly convex function under a Riemannian setting, i.e., every eigenvalue of the Riemannian Hessian of $f$ is positive, the global convergence of CGM with $\beta_{n}^{\rm HZ}$ is guaranteed~\cite{Sakai-Iiduka-Sato23}. However, such a strongly convexity is restricted for applications of Problem~\ref{problem:origin}. For example, the problem~\eqref{eq:eigenbasis} with $p=1$ and $\bm{A}=\bm{I}$ violates the strongly convexity.
	Indeed, the Riemannian Hessian
	$\mathop{\mathrm{Hess}}f(\bm{U})[\bm{\mathcal{D}}] = 2(\bm{U}^{\TT}\bm{\mathcal{D}})\bm{U}$~\cite[p.96]{Boumal20}
	implies
	$\inprod{\bm{\mathcal{D}}}{\mathop{\mathrm{Hess}}f(\bm{U})[\bm{\mathcal{D}}]} = 2(\bm{U}^{\TT}\bm{\mathcal{D}})^{2} = 0$
	for any
	$\bm{\mathcal{D}} \in T_{\bm{U}}\St(1,N)$
	because every
	$\bm{\mathcal{D}} \in T_{\bm{U}}\St(1,N)$
	can be expressed as
	$\bm{\mathcal{D}}=\bm{U}_{\perp}\bm{K}$
	with some
	$\bm{K} \in \mathbb{R}^{(N-1)\times 1}$,
	where
	$\bm{U}_{\perp} \in \St(N-1,N)$
	satisfies
	$\bm{U}^{\TT}\bm{U}_{\perp} = \bm{0}$.
}, their numerical performances may be practically superior to the CGM with
$\beta_{n}^{\rm FR}$~\cite{Sato21B}.

Our test problems are (i) eigenbasis extraction, e.g.,~\cite{Absil-Mahony-Sepulchre08}; (ii) unbalanced orthogonal Procrustes problem, e.g.,~\cite{Elden-Park99}.
For each problem, we generated an initial estimate randomly by MATLAB code 'orth(rand(N,p))'.

For a given symmetric matrix
$\bm{A} \in \mathbb{R}^{N\times N}$,
the eigenbasis extraction is formulated as
\begin{equation}
	\mathop{\mathrm{minimize}}_{\bm{U}\in\St(p,N)} f_2(\bm{U}):=-\trace(\bm{U}^{\TT}\bm{A}\bm{U}). \label{eq:eigenbasis}
\end{equation}
Any solution
$\bm{U}^{\star}$
of the problem~\eqref{eq:eigenbasis} is an orthonormal eigenbasis associated with the
$p$
largest eigenvalues of
$\bm{A}$~\cite{Horn-Johonson12}.
In our experiment, we used
$\bm{A}:=\widetilde{\bm{A}}^{\TT} \widetilde{\bm{A}} \in \mathbb{R}^{N\times N}$
with randomly chosen
$\widetilde{\bm{A}} \in \mathbb{R}^{N\times N}$
of which each entry is sampled by the standard normal distribution
$\mathcal{N}(0,1)$.

For given matrices
$\bm{B} \in \mathbb{R}^{M\times N}$
and
$\bm{C} \in \mathbb{R}^{M\times p}$,
the unbalanced orthogonal Procrustes problem is formulated with
$p < N$
as
\begin{equation}
	\mathop{\mathrm{minimize}}_{\bm{U}\in \St(p,N)} f_3(\bm{U}):=\|\bm{B}\bm{U}-\bm{C}\|_{F}^{2}. \label{eq:procrustes}
\end{equation}
Any closed-form solution to the problem~\eqref{eq:procrustes} has not been found~\cite{Zhang-Yang-Shen-Ying20}.
In our experiment, we used
$M = N$,
randomly chosen
$\bm{B} \in \mathbb{R}^{N\times N}$
of which each entry is sampled by
$\mathcal{N}(0,1)$,
and
$\bm{C}:=\bm{B}\bm{U}^{\star} \in \mathbb{R}^{N\times p}$
with randomly chosen
$\bm{U}^{\star} \in \St(p,N)$
by MATLAB code 'orth(rand(N,p))'.

For
$i \in \{2,3\}$,
since we can verify easily the Lipschitz continuity of
$\nabla f_{i}$,
for every
$\bm{S} \in {\rm O}(N)$,
$\nabla (f_{i} \circ \Phi_{\bm{S}}^{-1})$
is Lipschitz continuous over
$Q_{N,p}(\bm{S})$
with a common Lipschitz constant by Fact~\ref{fact:Lipschitz}.

Tables~\ref{table:eigenbasis} and~\ref{table:procrustes} illustrate average results for
$100$
trials of each algorithm for the problems~\eqref{eq:eigenbasis} and~\eqref{eq:procrustes}.
From Tables~\ref{table:eigenbasis} and~\ref{table:procrustes}, CGM(HS+)+ALCP outperformed the others in CPU time under both scenarios.
In contrast, regarding the number of iterations, CGM(HS+)+QR outperformed CGM(HS+)+ALCP slightly.
These imply that although CGM(HS+)+ALCP needed more iterations than CGM(HS+)+QR, the former converged in less CPU time than the latter.
These tendencies can apply to the other types of the CGM+ALCP compared to the CGM+QR respectively.

\begin{table}[t]
	\caption{Performance of each algorithm applied to the problem~\eqref{eq:eigenbasis}.}
	\centering
	\footnotesize
	\begin{filecontents*}{data/eigenvalue.csv}
Algorithm,fval,feasi,nrmg,itr,nfe,time,change
$N=1000$ and $p=1$,,,,,,,
\quad CGM(FR)+ALCP,3.79E-07,1.73E-16,1.18E-02,204.13,367.93,0.29,1.01
\quad CGM(HS+)+ALCP,3.80E-07,1.63E-16,1.15E-02,107.14,179.36,0.16,0.98
\quad CGM(HZ)+ALCP,1.20E-06,1.54E-16,1.14E-02,143.57,245.63,0.21,0.53
\quad CGM(FR)+QR,2.76E-07,2.11E-16,1.64E-02,179.71,326.91,1.35,-
\quad CGM(HS+)+QR,3.42E-07,2.19E-16,1.52E-02,93.50,164.82,0.69,-
\quad CGM(HZ)+QR,1.16E-06,2.38E-16,1.65E-02,128.87,223.92,1.12,-
$N=1000$ and $p=10$,,,,,,,
\quad CGM(FR)+ALCP,5.98E-06,8.77E-16,4.02E-02,277.35,481.81,1.31,2.29
\quad CGM(HS+)+ALCP,9.28E-06,8.45E-16,4.00E-02,164.96,256.76,0.75,2.84
\quad CGM(HZ)+ALCP,2.18E-05,8.79E-16,3.92E-02,226.61,365.93,1.08,2.04
\quad CGM(FR)+QR,5.41E-06,1.49E-15,5.12E-02,241.76,434.84,2.67,-
\quad CGM(HS+)+QR,9.54E-06,1.49E-15,5.08E-02,122.30,205.51,1.34,-
\quad CGM(HZ)+QR,2.11E-05,1.52E-15,5.32E-02,175.03,296.28,2.26,-
$N=2000$ and $p=1$,,,,,,,
\quad CGM(FR)+ALCP,7.72E-07,2.15E-16,2.38E-02,330.51,510.54,1.38,1.00
\quad CGM(HS+)+ALCP,1.07E-06,1.82E-16,2.28E-02,127.89,208.56,0.55,1.17
\quad CGM(HZ)+ALCP,2.32E-06,2.01E-16,2.37E-02,172.60,285.72,0.73,1.00
\quad CGM(FR)+QR,9.14E-07,2.37E-16,3.26E-02,215.18,389.53,2.26,-
\quad CGM(HS+)+QR,1.11E-06,2.41E-16,3.09E-02,112.94,193.73,1.17,-
\quad CGM(HZ)+QR,2.18E-06,2.22E-16,3.13E-02,152.39,264.18,1.77,-
$N=2000$ and $p=10$,,,,,,,
\quad CGM(FR)+ALCP,3.03E-05,1.03E-15,7.73E-02,359.13,630.52,4.11,2.15
\quad CGM(HS+)+ALCP,5.36E-05,1.04E-15,7.85E-02,201.73,312.53,2.20,2.51
\quad CGM(HZ)+ALCP,1.56E-04,1.01E-15,7.67E-02,282.92,459.16,3.20,1.21
\quad CGM(FR)+QR,4.04E-05,9.78E-16,1.05E-01,295.62,526.29,5.68,-
\quad CGM(HS+)+QR,3.25E-05,1.01E-15,9.72E-02,159.61,263.55,2.99,-
\quad CGM(HZ)+QR,1.19E-04,9.36E-16,1.02E-01,225.43,377.31,4.74,-
\end{filecontents*}
\csvautobooktabular{data/eigenvalue.csv}
	\label{table:eigenbasis}
\end{table}
\begin{table}[t]
	\caption{Performance of each algorithm applied to the problem~\eqref{eq:procrustes}.}
	\centering
	\footnotesize
	\begin{filecontents*}{data/procrustes.csv}
Algorithm,fval,feasi,nrmg,itr,nfe,time,change
$N=1000$ and $p=1$,,,,,,,
\quad CGM(FR)+ALCP,6.28E-04,1.44E-16,2.18E-02,849.77,1488.89,1.66,1.00
\quad CGM(HS+)+ALCP,6.49E-04,1.69E-16,2.20E-02,483.75,752.25,0.93,0.99
\quad CGM(HZ)+ALCP,1.07E-03,1.42E-16,2.32E-02,630.07,1011.82,1.24,0.95
\quad CGM(FR)+QR,5.71E-04,2.50E-16,2.93E-02,871.22,1536.17,6.93,-
\quad CGM(HS+)+QR,6.30E-04,2.30E-16,3.04E-02,471.88,742.39,3.69,-
\quad CGM(HZ)+QR,1.05E-03,2.01E-16,3.24E-02,610.88,991.11,5.55,-
$N=1000$ and $p=10$,,,,,,,
\quad CGM(FR)+ALCP,9.10E-03,9.19E-16,7.46E-02,1022.98,1796.3,6.74,1.14
\quad CGM(HS+)+ALCP,9.16E-03,9.01E-16,7.26E-02,585.38,906.27,3.72,1.00
\quad CGM(HZ)+ALCP,1.44E-02,9.31E-16,7.81E-02,764.91,1226.02,5.00,1.00
\quad CGM(FR)+QR,7.43E-03,1.55E-15,9.58E-02,941.77,1658.42,12.02,-
\quad CGM(HS+)+QR,7.87E-03,1.54E-15,9.54E-02,499.31,780.95,6.30,-
\quad CGM(HZ)+QR,1.30E-02,1.57E-15,1.02E-01,669.05,1077.99,9.84,-
$N=2000$ and $p=1$,,,,,,,
\quad CGM(FR)+ALCP,1.19E-03,1.92E-16,4.27E-02,877.61,1541.67,6.14,1.00
\quad CGM(HS+)+ALCP,1.24E-03,1.82E-16,4.54E-02,485.89,756.19,3.35,0.98
\quad CGM(HZ)+ALCP,2.09E-03,1.75E-16,4.64E-02,626.19,1009.64,4.30,0.77
\quad CGM(FR)+QR,1.32E-03,2.39E-16,5.88E-02,843.06,1486.19,11.06,-
\quad CGM(HS+)+QR,1.33E-03,2.39E-16,6.16E-02,473.07,745.35,6.08,-
\quad CGM(HZ)+QR,2.04E-03,2.23E-16,6.47E-02,610.92,992.59,8.67,-
$N=2000$ and $p=10$,,,,,,,
\quad CGM(FR)+ALCP,1.52E-02,1.05E-15,1.41E-01,972.80,1697.75,15.91,1.03
\quad CGM(HS+)+ALCP,1.61E-02,1.01E-15,1.47E-01,556.00,859.22,8.67,1.00
\quad CGM(HZ)+ALCP,2.58E-02,9.86E-16,1.55E-01,711.94,1138.97,11.45,1.00
\quad CGM(FR)+QR,1.39E-02,9.73E-16,1.88E-01,885.48,1554.08,21.16,-
\quad CGM(HS+)+QR,1.31E-02,9.75E-16,1.91E-01,489.97,770.05,11.35,-
\quad CGM(HZ)+QR,2.39E-02,9.80E-16,2.07E-01,636.25,1026.99,16.28,-
\end{filecontents*}
\csvautobooktabular{data/procrustes.csv}
	\label{table:procrustes}
\end{table}
\begin{figure}[t]
	\centering
	\subfloat[][$N=1000,p=1$]{
		\includegraphics[clip, width=0.4\columnwidth]{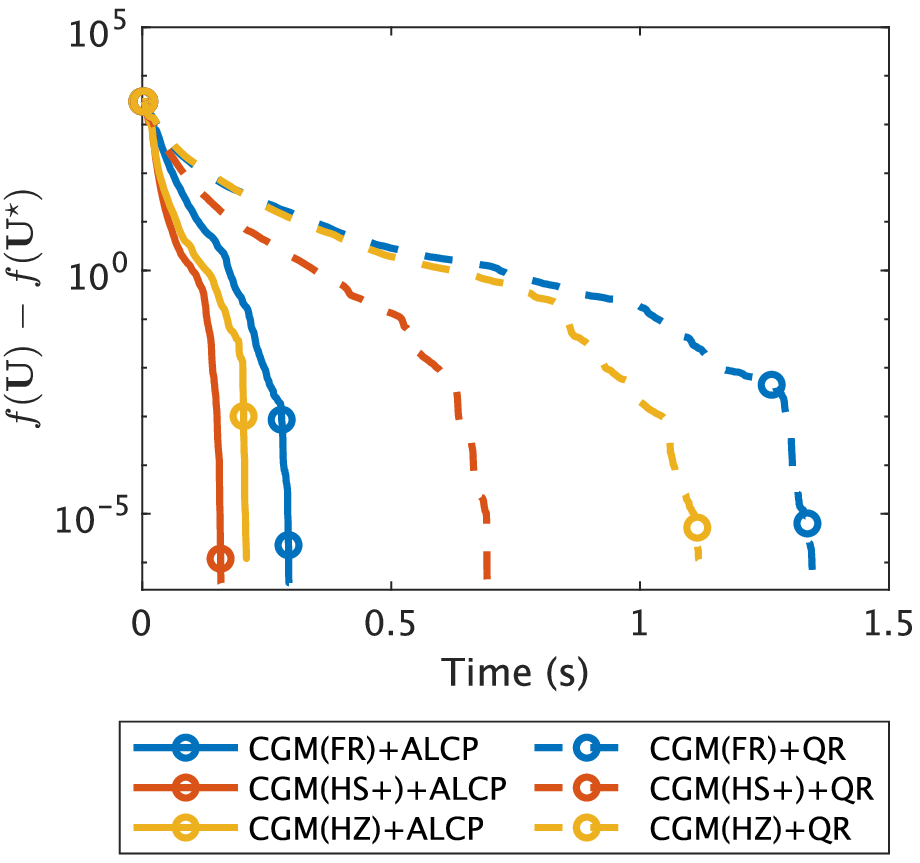}
	}
	\subfloat[][$N=1000,p=10$]{
		\includegraphics[clip, width=0.4\columnwidth]{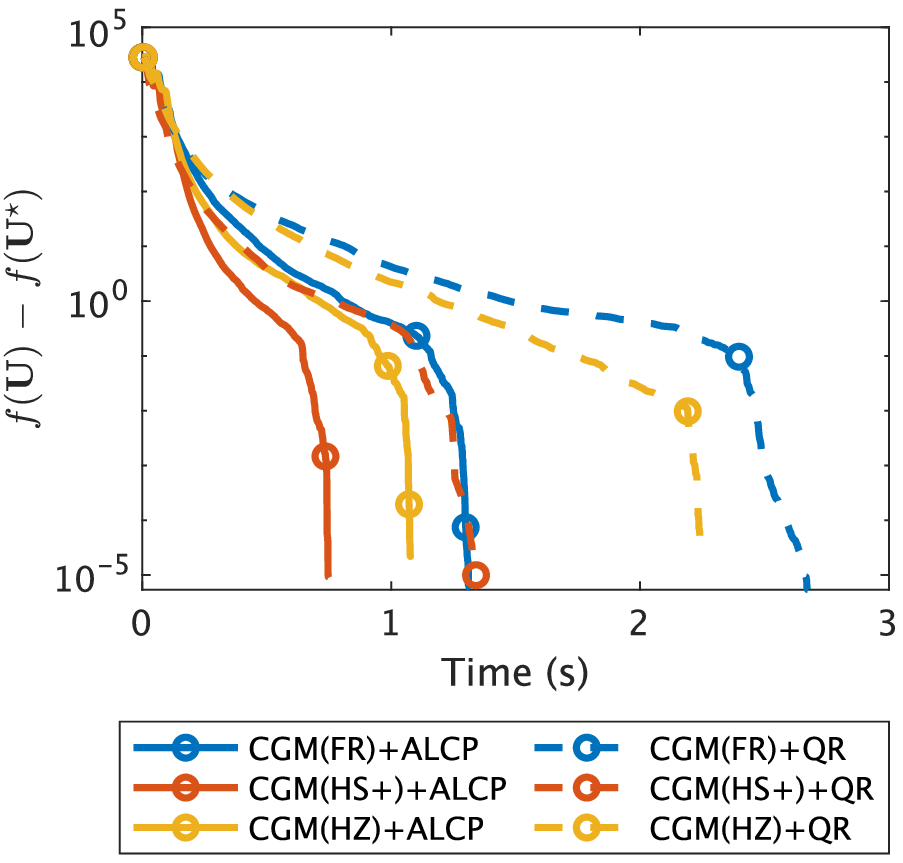}
	} \\
	\subfloat[][$N=2000,p=1$]{
		\includegraphics[clip, width=0.4\columnwidth]{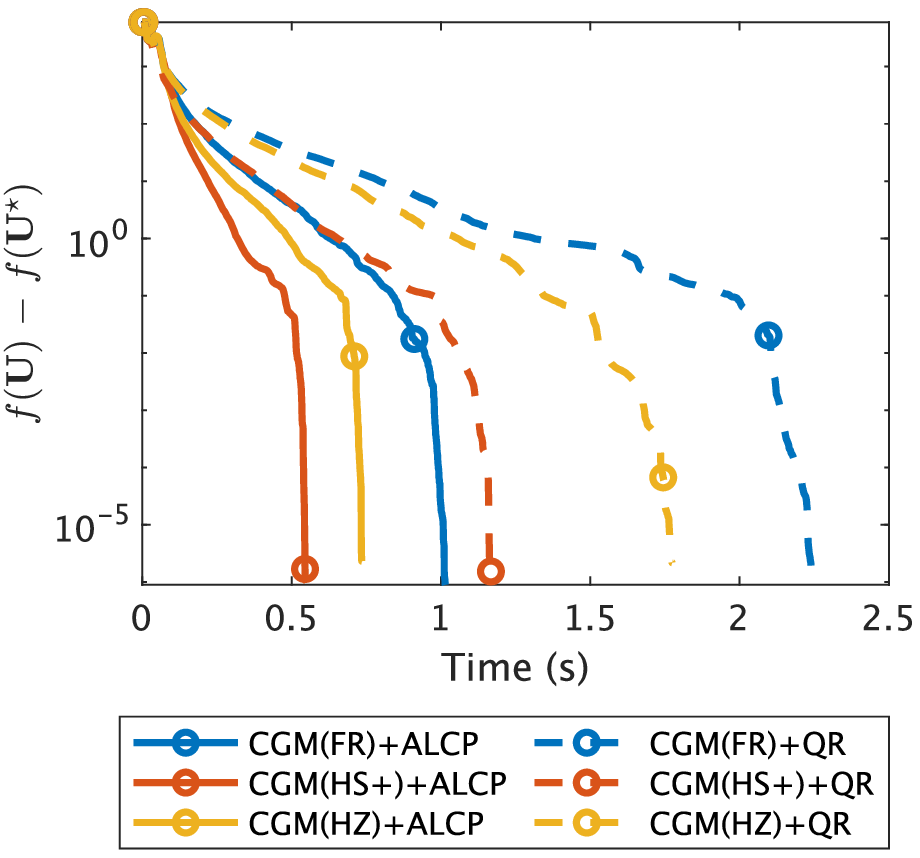}
	}
	\subfloat[][$N=2000,p=10$]{
		\includegraphics[clip, width=0.4\columnwidth]{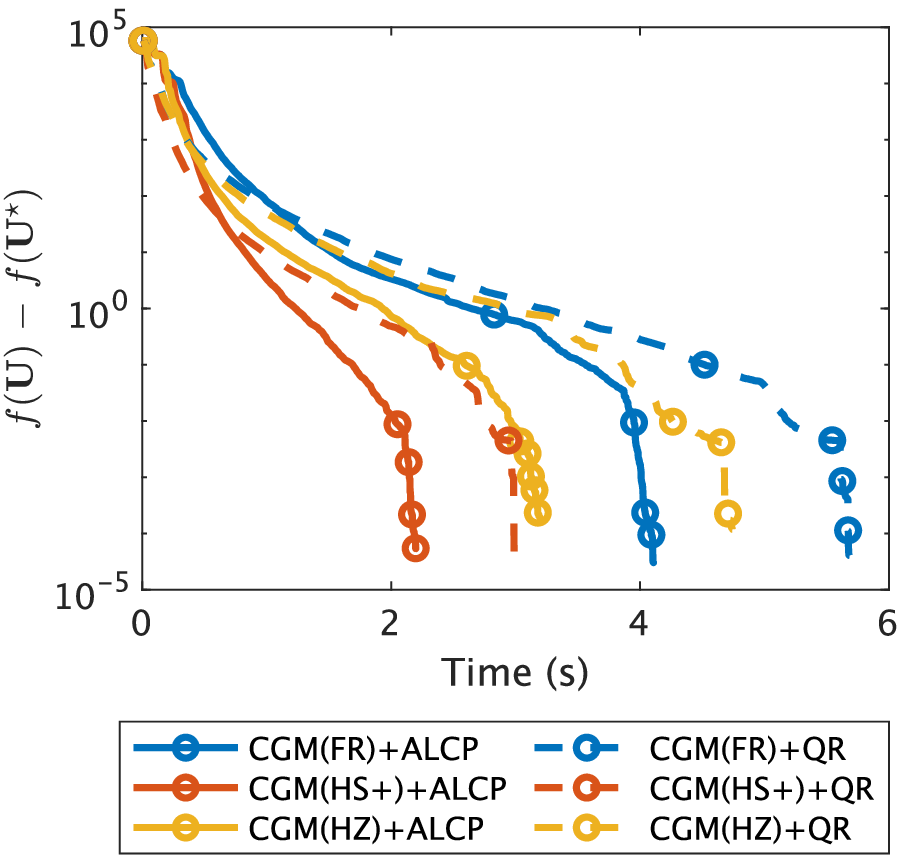}
	}
	\caption{Convergence histories of each algorithm applied to the problem~\eqref{eq:eigenbasis} regarding the value $f(\bm{U})-f(\bm{U}^{\star})$ at CPU time for each problem size, where
		$\bm{U}^{\star}\in\St(p,N)$
		was obtained by the eigenvalue decomposition of
		$\bm{A}$.
		Markers are put at every 250 iterations.}
	\label{fig:eigenbasis}
\end{figure}
\begin{figure}[t]
	\centering
	\subfloat[][$N=1000,p=1$]{
		\includegraphics[clip, width=0.4\columnwidth]{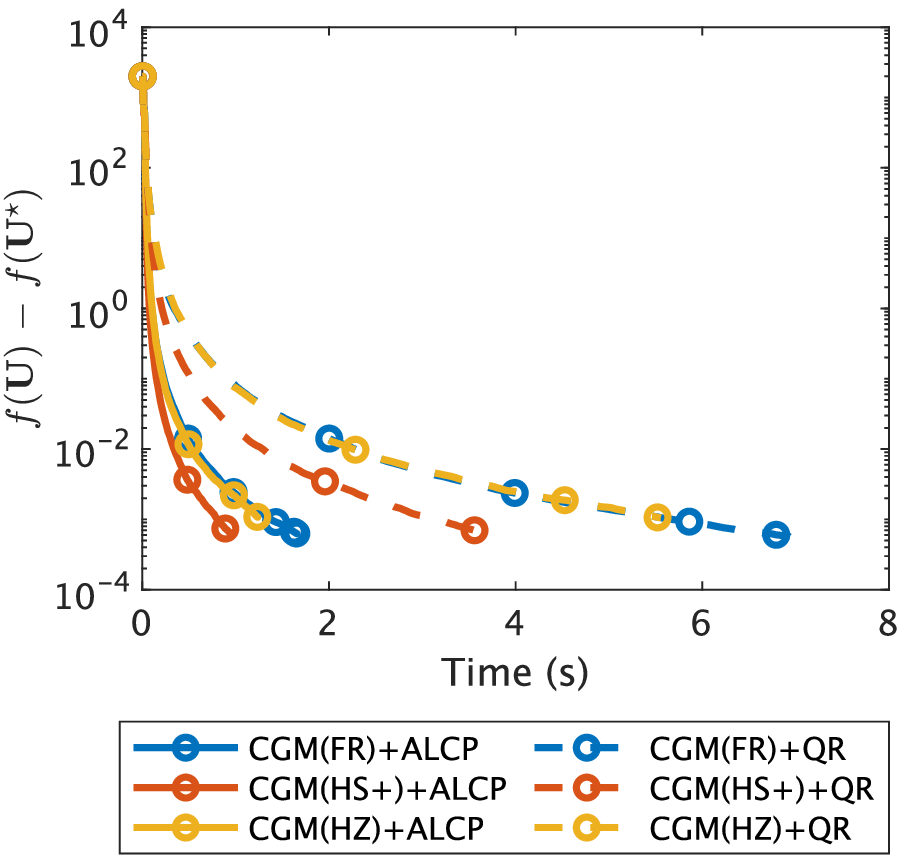}
	}
	\subfloat[][$N=1000,p=10$]{
		\includegraphics[clip, width=0.4\columnwidth]{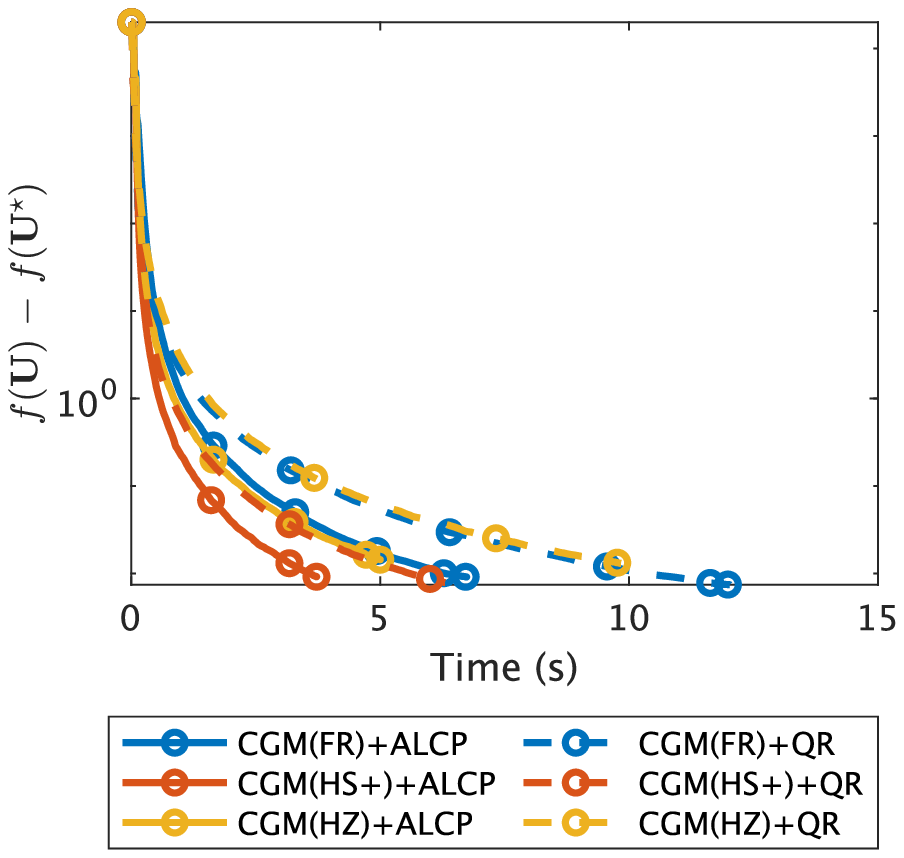}
	} \\
	\subfloat[][$N=2000,p=1$]{
		\includegraphics[clip, width=0.4\columnwidth]{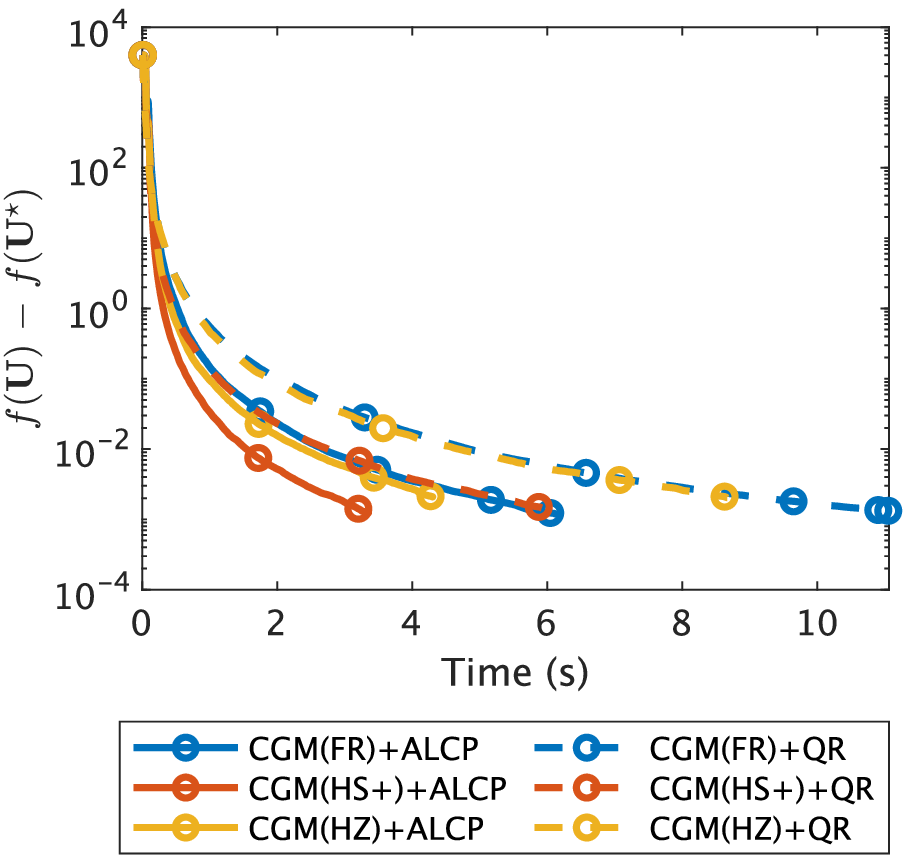}
	}
	\subfloat[][$N=2000,p=10$]{
		\includegraphics[clip, width=0.4\columnwidth]{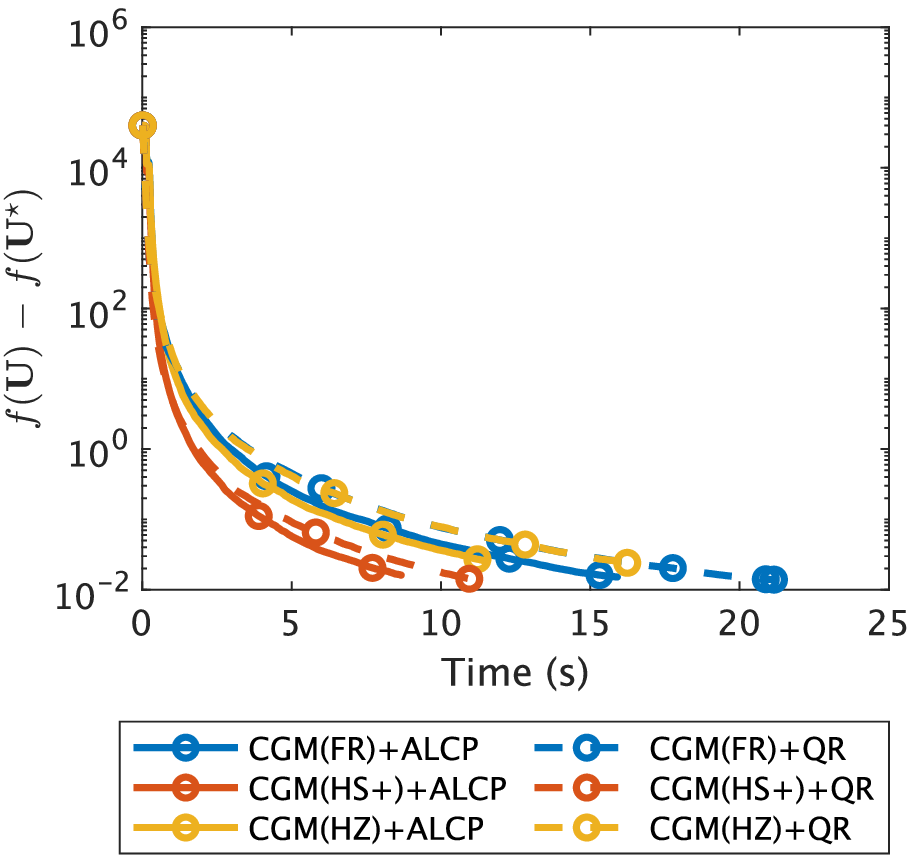}
	}
	\caption{Convergence histories of each algorithm applied to the problem~\eqref{eq:procrustes} regarding the value $f(\bm{U})-f(\bm{U}^{\star})$ at CPU time for each problem size.
		Markers are put at every 250 iterations.}
	\label{fig:procrustes}
\end{figure}

From Tables~\ref{table:eigenbasis} and~\ref{table:procrustes}, although the number of evaluations of
$f$
per iteration was the same level between CGM+ALCP and CGM+QR, CGM+QR needed more CPU time per iteration than CGM+ALCP.
Indeed, for each type of CGM, the average ratios of nfe/itr for CGM+ALCP and CGM+QR were between about 1.5 and 1.8, however, the average ratios of time/itr for CGM+QR were about 1.5-6 times higher than those for CGM+ALCP.
Since the computational complexity for the QR decomposition-based retraction and
$\Phi_{\bm{S}}^{-1}$
with
$\bm{S} \in {\rm O}_{p}(N)$
is
$2Np^{2} + \mathfrak{o}(p^{3})$
flops~\cite{Zhu17,Kume-Yamada22}, this difference regarding CPU time per iteration can be caused by computations of a vector transport in QR+CGM.
Moreover, for the ALCP strategy, since the average number of changing center points is small, e.g., 2.84 at most, we can see that the ALCP strategy rarely changes center points practically.
Combined with these observations for the ALCP strategy, the alarming condition~\eqref{eq:change} employed in line~\ref{lst:line:if} of Algorithm~\ref{alg:ALCP} can (i) detect the risk of singular-point issue; (ii) enjoy sufficiently potential of each Euclidean space
$Q_{N,p}(\bm{S}_{[l]})$
where the CGM is executed for minimization of
$f\circ \Phi_{\bm{S}_{[l]}}^{-1}$.

To compare convergence behaviors in detail, Figures~\ref{fig:eigenbasis} and~\ref{fig:procrustes} demonstrate the convergence histories of algorithms.
The plots show CPU time on the horizontal axis versus the value
$f(\bm{U})-f(\bm{U}^{\star})$
on the vertical axis.
From Figures~\ref{fig:eigenbasis} and~\ref{fig:procrustes}, we observed that ALCP+CGM(HS+) outperformed the others for all problems.

\section{Conclusion}
For optimization over the Stiefel manifold, we presented an adaptive reformulation strategy
by translating the original problem into optimization over a Euclidean space with the generalized Cayley transform.
The adaptive reformulation strategy can avoid a performance degradation appeared in the naive reformulation strategy caused by the singular-point of the generalized Cayley transform.
We also presented a unified convergence analysis for the proposed strategy when we use a fairly standard class of Euclidean optimization algorithms, e.g., the conjugate gradient method, and the quasi-Newton method.
Numerical experiments demonstrate that the proposed algorithms outperformed the standard algorithms designed with a retraction on the Stiefel manifold.

\section*{Funding}
This work was supported by JSPS Grants-in-Aid (19H04134) partially, by JSPS Grants-in-Aid (21J21353) and by JST SICORP (JPMJSC20C6).

\bibliographystyle{unsrt}
\bibliography{main}

\begin{thebibliography}{10}

\bibitem{Pietersz-Groenen04}
Raoul Pietersz and Patrick J~F Groenen.
\newblock Rank reduction of correlation matrices by majorization.
\newblock {\em Quantitative Finance}, 4(6):649--662, 2004.

\bibitem{Grubisic-Pietersz07}
Igor Grubi^^c5^^a1i^^c4^^87 and Raoul Pietersz.
\newblock Efficient rank reduction of correlation matrices.
\newblock {\em Linear Algebra and its Applications}, 422(2):629--653, 2007.

\bibitem{Zhu15}
Xiaojing Zhu.
\newblock A feasible filter method for the nearest low-rank correlation matrix
  problem.
\newblock {\em Numerical Algorithms}, 69(4):763--784, 2015.

\bibitem{Bai-Sleijpen-Vorst-Lippert-Edelman98}
Z.~Bai, G.~Sleijpen, H.~van~der Vorst, R.~Lippert, and A.~Edelman.
\newblock Nonlinear eigenvalue problems.
\newblock In Z.~Bai, J.~Demmel, J.~Dongarra, A.~Ruhe, and H.~van~der Vorst,
  editors, {\em Templates for the Solution of Algebraic Eigenvalue Problems},
  chapter~9, pages 281--314. SIAM, 2000.

\bibitem{Yang-Meza-Wang06}
Chao Yang, Juan~C. Meza, and Lin-Wang Wang.
\newblock A constrained optimization algorithm for total energy minimization in
  electronic structure calculations.
\newblock {\em Journal of Computational Physics}, 217(2):709--721, 2006.

\bibitem{Zhao-Bai-Jin15}
Zhi Zhao, Zheng-Jian Bai, and Xiao-Qing Jin.
\newblock A {R}iemannian {N}ewton algorithm for nonlinear eigenvalue problems.
\newblock {\em SIAM Journal on Matrix Analysis and Applications},
  36(2):752--774, 2015.

\bibitem{Joho-Mathis02}
M.~Joho and H.~Mathis.
\newblock Joint diagonalization of correlation matrices by using gradient
  methods with application to blind signal separation.
\newblock In {\em Sensor Array and Multichannel Signal Processing Workshop
  Proceedings}, pages 273--277. IEEE, 2002.

\bibitem{Theis-Cason-Absil09}
Fabian~J. Theis, Thomas~P. Cason, and P.~A. Absil.
\newblock Soft dimension reduction for {ICA} by joint diagonalization on the
  {S}tiefel manifold.
\newblock In {\em Independent Component Analysis and Signal Separation}, pages
  354--361. Springer, 2009.

\bibitem{Sato17}
Hiroyuki Sato.
\newblock Riemannian {N}ewton-type methods for joint diagonalization on the
  {S}tiefel manifold with application to independent component analysis.
\newblock {\em Optimization}, 66(12):2211--2231, 2017.

\bibitem{Nikpour02}
M.~{Nikpour}, J.~H. {Manton}, and G.~{Hori}.
\newblock Algorithms on the {S}tiefel manifold for joint diagonalisation.
\newblock In {\em International Conference on Acoustics, Speech, and Signal
  Processing}, volume~2, pages 1481--1484. IEEE, 2002.

\bibitem{Elden-Park99}
Lars Eld{\'e}n and Haesun Park.
\newblock A {P}rocrustes problem on the {S}tiefel manifold.
\newblock {\em Numerische Mathematik}, 82(4):599--619, 1999.

\bibitem{Francisco-Viloche-Weber2017}
J.B. Francisco, F.S. {Viloche Baz^^c3^^a1n}, and M.~{Weber Mendon^^c3^^a7a}.
\newblock Non-monotone algorithm for minimization on arbitrary domains with
  applications to large-scale orthogonal {P}rocrustes problem.
\newblock {\em Applied Numerical Mathematics}, 112:51--64, 2017.

\bibitem{Zhao-Wang-Nie16}
Haifeng Zhao, Zheng Wang, and Feiping Nie.
\newblock Orthogonal least squares regression for feature extraction.
\newblock {\em Neurocomputing}, 216:200--207, 2016.

\bibitem{Zhang-Yang-Shen-Ying20}
Lei-Hong Zhang, Wei~Hong Yang, Chungen Shen, and Jiaqi Ying.
\newblock An eigenvalue-based method for the unbalanced {P}rocrustes problem.
\newblock {\em SIAM Journal on Matrix Analysis and Applications},
  41(3):957--983, 2020.

\bibitem{Helfrich-Willmott-Ye18}
Kyle Helfrich, Devin Willmott, and Qiang Ye.
\newblock Orthogonal recurrent neural networks with scaled {C}ayley transform.
\newblock In {\em International Conference on Machine Learning}, volume~80,
  pages 1969--1978. PMLR, 2018.

\bibitem{Bansal-Chen-Wang18}
Nitin Bansal, Xiaohan Chen, and Zhangyang Wang.
\newblock Can we gain more from orthogonality regularizations in training deep
  networks?
\newblock In {\em Advances in Neural Information Processing Systems}, pages
  4266--4276. Curran Associates Inc., 2018.

\bibitem{Yamada-Ezaki03}
Isao Yamada and Takato Ezaki.
\newblock An orthogonal matrix optimization by dual {C}ayley parametrization
  technique.
\newblock In {\em 4th International Symposium on Independent Component Analysis
  and Blind Signal Separation}, pages 35--40, 2003.

\bibitem{Fraikin-Huper-Dooren07}
C.~Fraikin, K.~H^^c3^^bcper, and P.~Van Dooren.
\newblock Optimization over the {S}tiefel manifold.
\newblock In {\em Proceedings in Applied Mathematics and Mechanics}, volume~7.
  Wiley, 2007.

\bibitem{Hori-Tanaka10}
Gen Hori and Toshihisa Tanaka.
\newblock Pivoting in {C}ayley tranform-based optimization on orthogonal
  groups.
\newblock In {\em Asia Pacific Signal and Information Processing Association
  Annual Summit and Conference}, pages 181--184, 2010.

\bibitem{Kume-Yamada19}
K.~Kume and I.~Yamada.
\newblock Adaptive localized {C}ayley parametrization technique for smooth
  optimization over the {S}tiefel manifold.
\newblock In {\em European Signal Processing Conference}, pages 500--504.
  EURASIP, 2019.

\bibitem{Kume-Yamada20}
K.~Kume and I.~Yamada.
\newblock A {N}esterov-type acceleration with adaptive localized {C}ayley
  parametrization for optimization over the {S}tiefel manifold.
\newblock In {\em European Signal Processing Conference}, pages 2105--2109.
  EURASIP, 2020.

\bibitem{Kume-Yamada21}
Keita Kume and Isao Yamada.
\newblock A global {C}ayley parametrization of {S}tiefel manifold for direct
  utilization of optimization mechanisms over vector spaces.
\newblock In {\em International Conference on Acoustics, Speech, and Signal
  Processing}, pages 5554--5558. IEEE, 2021.

\bibitem{Kume-Yamada22}
Keita Kume and Isao Yamada.
\newblock Generalized left-localized {C}ayley parametrization for optimization
  with orthogonality constraints.
\newblock {\em Optimization}, 0(0):1--47, 2022.

\bibitem{Absil-Mahony-Sepulchre08}
P.-A. Absil, R.~Mahony, and R.~Sepulchre.
\newblock {\em Optimization Algorithms on Matrix Manifolds}.
\newblock Princeton University Press, Princeton (NJ), 2008.

\bibitem{Wen-Yin13}
Zaiwen Wen and Wotao Yin.
\newblock A feasible method for optimization with orthogonality constraints.
\newblock {\em Mathematical Programming}, 142(1):397--434, 2013.

\bibitem{Gao-Liu-Chen-Yuan18}
Bin Gao, Xin Liu, Xiaojun Chen, and Ya-Xiang Yuan.
\newblock A new first-order algorithmic framework for optimization problems
  with orthogonality constraints.
\newblock {\em SIAM Journal on Optimization}, 28(1):302--332, 2018.

\bibitem{Nocedal-Wright06}
Jorge Nocedal and Stephen Wright.
\newblock {\em Numerical optimization}.
\newblock Springer, New York (NY), 2006.

\bibitem{Li-Fukushima01}
Dong-Hui Li and Masao Fukushima.
\newblock On the global convergence of the {BFGS} method for nonconvex
  unconstrained optimization problems.
\newblock {\em SIAM Journal on Optimization}, 11(4):1054--1064, 2001.

\bibitem{Andrei20}
Neculai Andrei.
\newblock {\em Nonlinear conjugate gradient methods for unconstrained
  optimization}.
\newblock Springer, New York (NY), 2020.

\bibitem{Gilbert-Nocedal92}
Jean~Charles Gilbert and Jorge Nocedal.
\newblock Global convergence properties of conjugate gradient methods for
  optimization.
\newblock {\em SIAM Journal on Optimization}, 2(1):21--42, 1992.

\bibitem{Al-baali85}
M.~Al-Baali.
\newblock Descent property and global convergence of the {Fletcher―Reeves}
  method with inexact line search.
\newblock {\em IMA Journal of Numerical Analysis}, 5(1):121--124, 1985.

\bibitem{Dai-Yuan99}
Y.~Dai and Y.~Yuan.
\newblock A nonlinear conjugate gradient method with a strong global
  convergence property.
\newblock {\em SIAM Journal on Optimization}, 10(1):177--182, 1999.

\bibitem{Dai-etal00}
Yuhong Dai, Jiye Han, Guanghui Liu, Defeng Sun, Hongxia Yin, and Ya-Xiang Yuan.
\newblock Convergence properties of nonlinear conjugate gradient methods.
\newblock {\em SIAM Journal on Optimization}, 10(2):345--358, 2000.

\bibitem{Dai-Yuan01}
Y.~H. Dai and Y.~Yuan.
\newblock An efficient hybrid conjugate gradient method for unconstrained
  optimization.
\newblock {\em Annals of Operations Research}, 103(1):33--47, 2001.

\bibitem{Hager-Zhang05}
William~W. Hager and Hongchao Zhang.
\newblock A new conjugate gradient method with guaranteed descent and an
  efficient line search.
\newblock {\em SIAM Journal on Optimization}, 16(1):170--192, 2005.

\bibitem{Zhang-Zhou-Li07}
Li~Zhang, Weijun Zhou, and Donghui Li.
\newblock Some descent three-term conjugate gradient methods and their global
  convergence.
\newblock {\em Optimization Methods and Software}, 22(4):697--711, 2007.

\bibitem{Narushima-Yabe-Ford11}
Yasushi Narushima, Hiroshi Yabe, and John~A Ford.
\newblock A three-term conjugate gradient method with sufficient descent
  property for unconstrained optimization.
\newblock {\em SIAM Journal on Optimization}, 21(1):212--230, 2011.

\bibitem{Khoshsimaye-Ashrafi23}
Maryam Khoshsimaye-Bargard and Ali Ashrafi.
\newblock A family of the modified three-term {H}estenes^^e2^^80^^93{S}tiefel
  conjugate gradient method with sufficient descent and conjugacy conditions.
\newblock {\em Journal of Applied Mathematics and Computing}, 2023.

\bibitem{Nesterov83}
Yurii Nesterov.
\newblock A method for solving the convex programming problem with convergence
  rate $o(1/k^2)$.
\newblock {\em Dokl Akad Nauk {SSSR}}, 269:543--547, 1983.

\bibitem{Ghadimi-Lan16}
Saeed Ghadimi and Guanghui Lan.
\newblock Accelerated gradient methods for nonconvex nonlinear and stochastic
  programming.
\newblock {\em Mathematical Programming}, 156(1):59--99, 2016.

\bibitem{Carmon-Duchi-Hinder-Sidford18}
Yair. Carmon, John~C. Duchi, Oliver. Hinder, and Aaron. Sidford.
\newblock Accelerated methods for nonconvex optimization.
\newblock {\em SIAM Journal on Optimization}, 28(2):1751--1772, 2018.

\bibitem{Zeyuan18}
Z.~Allen-Zhu.
\newblock Natasha 2: Faster non-convex optimization than {SGD}.
\newblock In {\em Advances in Neural Information Processing Systems}, pages
  2680--2691. Curran Associates Inc., 2018.

\bibitem{Diakonikolas-Jordan21}
Jelena Diakonikolas and Michael~I. Jordan.
\newblock Generalized momentum-based methods: A hamiltonian perspective.
\newblock {\em SIAM Journal on Optimization}, 31(1):915--944, 2021.

\bibitem{Lezcano19}
M.~Lezcano-Casado.
\newblock Trivializations for gradient-based optimization on manifolds.
\newblock In {\em Advances in Neural Information Processing Systems}, pages
  9157--9168. Curran Associates Inc., 2019.

\bibitem{Criscitiello-Boumal19}
C.~Criscitiello and N.~Boumal.
\newblock Efficiently escaping saddle points on manifolds.
\newblock In {\em Advances in Neural Information Processing Systems}, pages
  5987--5997. Curran Associates Inc., 2019.

\bibitem{Lezcano20}
Mario Lezcano-Casado.
\newblock Curvature-dependant global convergence rates for optimization on
  manifolds of bounded geometry, 2020.

\bibitem{Criscitiello-Boumal22}
Christopher Criscitiello and Nicolas Boumal.
\newblock An accelerated first-order method for non-convex optimization on
  manifolds.
\newblock {\em Foundations of Computational Mathematics}, 2022.

\bibitem{Horn-Johonson12}
Roger~A Horn and Charles~R Johnson.
\newblock {\em Matrix analysis}.
\newblock Cambridge university press, Cambridge (MA), 2nd edition, 2012.

\bibitem{Boumal20}
Nicolas Boumal.
\newblock {\em An introduction to optimization on smooth manifolds}.
\newblock 2020.

\bibitem{Edelman-Arias-Smith98}
Alan Edelman, Tom^^c3^^a1s~A. Arias, and Steven~T. Smith.
\newblock The geometry of algorithms with orthogonality constraints.
\newblock {\em SIAM Journal on Matrix Analysis and Applications},
  20(2):303--353, 1998.

\bibitem{Nishimori-Akaho05}
Yasunori Nishimori and Shotaro Akaho.
\newblock Learning algorithms utilizing quasi-geodesic flows on the {S}tiefel
  manifold.
\newblock {\em Neurocomputing}, 67:106--135, 2005.

\bibitem{Abrudan-Eriksson-Koivunen}
T.~E. {Abrudan}, J.~{Eriksson}, and V.~{Koivunen}.
\newblock Steepest descent algorithms for optimization under unitary matrix
  constraint.
\newblock {\em IEEE Transactions on Signal Processing}, 56(3):1134--1147, 2008.

\bibitem{Manton15}
Jonathan~H. Manton.
\newblock A framework for generalising the {N}ewton method and other iterative
  methods from {E}uclidean space to manifolds.
\newblock {\em Numerische Mathematik}, 129:91--125, 2015.

\bibitem{Absil-Baker-Gallivan07}
P.-A. Absil, C.~G. Baker, and K.~A. Gallivan.
\newblock Trust-region methods on {R}iemannian manifolds.
\newblock {\em Foundations of Computational Mathematics},
  7(3):303--^^e2^^80^^93330, 2007.

\bibitem{Kasai-Mishra18}
Hiroyuki Kasai and Bamdev Mishra.
\newblock Inexact trust-region algorithms on {R}iemannian manifolds.
\newblock In {\em Advances in Neural Information Processing Systems}, pages
  4254--4265. Curran Associates Inc., 2018.

\bibitem{Ring-Wirth12}
Wolfgang Ring and Benedikt Wirth.
\newblock Optimization methods on {R}iemannian manifolds and their application
  to shape space.
\newblock {\em SIAM Journal on Optimization}, 22(2):596--627, 2012.

\bibitem{Sato-Iwai15}
Hiroyuki Sato and Toshihiro Iwai.
\newblock A new, globally convergent {R}iemannian conjugate gradient method.
\newblock {\em Optimization}, 64(4):1011--1031, 2015.

\bibitem{Zhu17}
Xiaojing Zhu.
\newblock A {R}iemannian conjugate gradient method for optimization on the
  {S}tiefel manifold.
\newblock {\em Computational Optimization and Applications}, 67(1):73--110,
  2017.

\bibitem{Zhu-Sato20}
Xiaojing Zhu and Hiroyuki Sato.
\newblock Riemannian conjugate gradient methods with inverse retraction.
\newblock {\em Computational Optimization and Applications}, 77(3):779--810,
  2020.

\bibitem{Sakai-Iiduka20}
Hiroyuki Sakai and Hideaki Iiduka.
\newblock Hybrid {R}iemannian conjugate gradient methods with global
  convergence properties.
\newblock {\em Computational Optimization and Applications}, 77(3):811--830,
  2020.

\bibitem{Sato21B}
Hiroyuki Sato.
\newblock Riemannian conjugate gradient methods: General framework and specific
  algorithms with convergence analyses.
\newblock {\em SIAM Journal on Optimization}, 32(4):2690--2717, 2022.

\bibitem{Sakai-Iiduka-Sato23}
Hiroyuki Sakai, Hiroyuki Sato, and Hideaki Iiduka.
\newblock Global convergence of {H}ager^^e2^^80^^93{Z}hang type {R}iemannian
  conjugate gradient method.
\newblock {\em Applied Mathematics and Computation}, 441:127685, 2023.

\bibitem{Huag-Gallivan-Absil15}
Wen. Huang, K.~A. Gallivan, and P.-A. Absil.
\newblock A {B}royden class of quasi-{N}ewton methods for {R}iemannian
  optimization.
\newblock {\em SIAM Journal on Optimization}, 25(3):1660--1685, 2015.

\bibitem{Al-Higham09}
Awad~H. Al-Mohy and Nicholas~J. Higham.
\newblock Computing the {F}r^^c3^^a9chet derivative of the matrix exponential,
  with an application to condition number estimation.
\newblock {\em SIAM Journal on Matrix Analysis and Applications},
  30(4):1639--1657, 2009.

\bibitem{Boumal-Mishra-Absil-Sepulchre14}
Nicolas Boumal, Bamdev Mishra, P.-A. Absil, and Rodolphe Sepulchre.
\newblock Manopt, a {M}atlab toolbox for optimization on manifolds.
\newblock {\em Journal of Machine Learning Research}, 15:1455--1459, 2014.

\end{thebibliography}
\appendix

\section{Gradient of function after the Cayley parametrization} \label{appendix:gradient}
The gradient of
$f\circ\Phi_{\bm{S}}^{-1}$
is explicitly given by the following Fact~\ref{fact:gradient}.
Fact~\ref{fact:Lipschitz} below presents a sufficient condition for the Lipschitz continuity of
$\nabla (f\circ\Phi_{\bm{S}}^{-1})$.
\begin{fact}[Gradient of function after the Cayley parametrization\cite{Kume-Yamada22}] \label{fact:gradient}
	For a differentiable function
	$f:\mathbb{R}^{N\times p}\to \mathbb{R}$
	and
	$\bm{S} \in {\rm O}(N)$,
	the function
	$f_{\bm{S}}:=f\circ\Phi_{\bm{S}}^{-1}:Q_{N,p}(\bm{S}) \to \mathbb{R}$
	is differentiable with
	\begin{align}
		(\bm{V}\in Q_{N,p}(\bm{S})) \quad \nabla f_{\bm{S}}(\bm{V}) = 2\Skew(\bm{W}^{f}_{\bm{S}}(\bm{V})) = \bm{W}^{f}_{\bm{S}}(\bm{V}) - \bm{W}^{f}_{\bm{S}}(\bm{V})^{\TT}\in Q_{N,p}(\bm{S}), \label{eq:grad_propo}
	\end{align}
	where
	\begin{equation}
		\bm{W}^{f}_{\bm{S}}(\bm{V})
		:= \begin{bmatrix}
			\dbra{\overline{\bm{W}}_{\bm{S}}^{f}(\bm{V})}_{11} & \dbra{\overline{\bm{W}}_{\bm{S}}^{f}(\bm{V})}_{12} \\ 
			\dbra{\overline{\bm{W}}_{\bm{S}}^{f}(\bm{V})}_{21} & \bm{0}
		\end{bmatrix} \in \mathbb{R}^{N\times N} \label{eq:matrix_W}
	\end{equation}
	and
	{
	\thickmuskip=0.0\thickmuskip
	\medmuskip=0.0\medmuskip
	\thinmuskip=0.0\thinmuskip
	\begin{align}
		& \overline{\bm{W}}^{f}_{\bm{S}}(\bm{V})
		:= (\bm{I}+\bm{V})^{-1}\bm{I}_{N\times p}\nabla f(\Phi_{\bm{S}}^{-1}(\bm{V}))^{\TT}\bm{S}(\bm{I}+\bm{V})^{-1} \label{eq:matrix_W_bar} \\
		= & \begin{bmatrix}
			\bm{M}^{-1}\nabla f(\bm{U})^{\TT}(\bm{S}_{\rm le} -\bm{S}_{\rm ri}\dbra{\bm{V}}_{21})\bm{M}^{-1}
			&
			\bm{M}^{-1}
			\nabla f(\bm{U})^{\TT}
			((\bm{S}_{\rm le} -\bm{S}_{\rm ri}\dbra{\bm{V}}_{21})\bm{M}^{-1}\dbra{\bm{V}}_{21}^{\TT} +\bm{S}_{\rm ri})
			\\
			-\dbra{\bm{V}}_{21}\bm{M}^{-1}\nabla f(\bm{U})^{\TT}(\bm{S}_{\rm le} -\bm{S}_{\rm ri}\dbra{\bm{V}}_{21})\bm{M}^{-1}
			&
			-\dbra{\bm{V}}_{21}\bm{M}^{-1}
			\nabla f(\bm{U})^{\TT}
			((\bm{S}_{\rm le} -\bm{S}_{\rm ri}\dbra{\bm{V}}_{21})\bm{M}^{-1}\dbra{\bm{V}}_{21}^{\TT} +\bm{S}_{\rm ri})
		\end{bmatrix}
	\end{align}
}%
	in terms of
	$\bm{U}:=\Phi_{\bm{S}}^{-1}(\bm{V}) \in \St(p,N)$
	and
	$\bm{M}:=\bm{I}_{p}+\dbra{\bm{V}}_{11}+\dbra{\bm{V}}_{21}^{\TT}\dbra{\bm{V}}_{21}\in \mathbb{R}^{p\times p}$.
\end{fact}

\begin{fact}[\cite{Kume-Yamada22}] \label{fact:Lipschitz}
	Let
	$f:\mathbb{R}^{N\times p}\to\mathbb{R}$
	be continuously differentiable.
	If it holds that
	\begin{equation}
		(\exists L> 0, \forall \bm{U}_1,\bm{U}_2 \in \St(p,N)) \quad \|\nabla f(\bm{U}_1)-\nabla f(\bm{U}_2)\|_{F} \leq L \|\bm{U}_1 - \bm{U}_2\|_{F}
	\end{equation}
	and
	$\mu \geq \max_{\bm{U}\in \St(p,N)} \|\nabla f(\bm{U})\|_{2}$,
	then we have
	\begin{equation}
		(\forall \bm{S} \in {\rm O}(N),\forall \bm{V}_1,\bm{V}_2 \in Q_{N,p}(\bm{S})) \quad\| \nabla f_{\bm{S}}(\bm{V}_1) - \nabla f_{\bm{S}}(\bm{V}_2) \|_{F}
		\leq 4(\mu + L) \|\bm{V}_1-\bm{V}_2\|_{F}. \label{eq:prop:gradient_Lipschitz}
	\end{equation}
\end{fact}

\end{document}